\documentclass[a4paper,10pt,leqno]{article}
\usepackage{amsmath,amssymb,amscd,amsthm}
\usepackage{epic,eepic}
\usepackage[dvips]{graphics}
\DeclareFontEncoding{OT2}{}{}
\DeclareFontSubstitution{OT2}{wncyr}{m}{n}
\DeclareSymbolFont{cyss}{OT2}{wncyss}{m}{n}
\DeclareSymbolFont{cyr}{OT2}{wncyr}{m}{n}
\DeclareMathSymbol{\sh}{\mathbin}{cyss}{`x}

\newcommand{\id}{\operatorname{id}}

\newcommand{\End}{\operatorname{End}}
\newcommand{\ad}{\operatorname{ad}}
\newcommand{\PGL}{\operatorname{PGL}}
\newcommand{\cMPL}{\mathcal{M\!P\!L}}
\newcommand{\cIDX}{\mathcal{I\!D\!X}}

\newcommand{\Li}{\operatorname{Li}}

\renewcommand{\Pr}{\operatorname{Pr}}

\newcommand{\C}{{\mathbf C}}
\newcommand{\R}{{\mathbf R}}
\newcommand{\Q}{{\mathbf Q}}
\newcommand{\Z}{{\mathbf Z}}

\newcommand{\bP}{{\mathbf P}}
\newcommand{\bunit}{{\mathbf I}}
\newcommand{\bnull}{{\mathbf 1}}
\newcommand{\F}{\mathbf F}

\newcommand{\bk}{{\mathbf k}}
\newcommand{\bl}{{\mathbf l}}

\newcommand{\fd}{{\mathfrak d}}
\newcommand{\fg}{{\mathfrak g}}

\newcommand{\fH}{{\mathfrak H}}

\newcommand{\fI}{{\mathfrak I}}
\newcommand{\fJ}{{\mathfrak J}}

\newcommand{\fX}{{\mathfrak X}}

\newcommand{\fM}{{\mathfrak M}}
\newcommand{\fA}{{\mathfrak A}}

\newcommand{\cB}{{\mathcal B}}
\newcommand{\cL}{{\mathcal L}}
\newcommand{\cM}{{\mathcal M}}
\newcommand{\cN}{{\mathcal N}}
\newcommand{\cU}{{\mathcal U}}
\newcommand{\cW}{{\mathcal W}}
\newcommand{\hcL}{{\widehat{\mathcal L}}}
\newcommand{\hG}{\widehat{G}}

\newcommand{\ds}{\displaystyle}

\newtheorem{thm}{Theorem}
\newtheorem{cor}[thm]{Corollary}
\newtheorem{prop}[thm]{Proposition}
\newtheorem{lem}[thm]{Lemma}

\title{KZ equation on the moduli space ${\mathcal M}_{0,5}$ and\\ the harmonic product of multiple polylogarithms}
\def\subjclass{Primary 32G34; Secondary 11G55,11M06;}
\author{OI, Shu and UENO, Kimio}
\date{}
\def\address{\noindent
{\bf OI, Shu.}\\
Department of Mathematics, School of Fundamental Sciences and Engineering, Faculty of Science and Engineering, Waseda university. 3-4-1, Okubo, Shinjuku-ku, Tokyo 169-8555, Japan.\\
{\it e-mail:} {\tt shu\_oi@toki.waseda.jp}\\[1\baselineskip]
{\bf UENO, Kimio}\\
Department of Mathematics, School of Fundamental Sciences and Engineering, Faculty of Science and Engineering, Waseda university. 3-4-1, Okubo, Shinjuku-ku, Tokyo 169-8555, Japan.\\
{\it e-mail:} {\tt uenoki@waseda.jp}
}

\begin{document}
\maketitle
\insert\footins{\footnotesize 2010 {\it Mathematics Subject Classification.} \subjclass}

\begin{abstract}
In this article, we derive a system of functional relations called the generalized harmonic product relations for hyperlogarithms on the moduli space ${\mathcal M}_{0,5}$ and show that the relations contain the harmonic product of multiple polylogarithms. The generalized harmonic product relations are equivalent to the relations which come from two decompositions of the fundamental solution normalized at the origin of the KZ equation on ${\mathcal M}_{0,5}$.\\
\end{abstract}

\section{Introduction}
The Knizhnik-Zamolodchikov equation (the KZ equation, KZE, for short) is a differential equation on the moduli space $\cM_{0,n}$ whose coefficients are generators of the infinitesimal pure braid Lie algebra of $\cM_{0,n}$.

Our aim is deriving the harmonic product of multiple zeta values (MZVs) through studies on the KZ equation on $\cM_{0,5}$, which is regarded as a differential equation of two variables. Deligne-Terasoma \cite{DT} and Furusho \cite{F} succeeded in deriving the harmonic product of MZVs from relations of Drinfel'd associator through the viewpoint of arithmetic geometry. Drinfel'd associator appears as a connection matrix between certain solutions of KZE.

In this paper we consider KZE on $\cM_{0,5}$ and its fundamental solutions. From the viewpoint of differential equations and iterated integrals, we introduce the generalized harmonic product relations. They are functional relations for hyperlogarithms which contain the harmonic product of multiple polylogarithms (MPLs). 

The central idea of our method is the decomposition of the infinitesimal pure braid Lie algebra $\fX$, the reduced bar algebra $\cB$, which is a dual Hopf algebra of $\cU(\fX)$, and fundamental solutions of KZE on the moduli space $\cM_{0,5}$. Brown \cite{B} mentioned to the decomposition of the KZ equation on the moduli space, however, we will consider such decompositions concretely from the topological and algebraic viewpoints.

Since the harmonic product of MPLs yields the harmonic product of MZVs as a boundary values, the result of this paper says that the harmonic product of MZVs comes from the decomposition of the fundamental solution of the KZ equation on $\cM_{0,5}$.\\

This paper is organized as follows: Section 2 is devoted to the preliminaries for universal enveloping algebras of Lie algebras and shuffle algebras. In Section 3 and 4, we introduce the moduli space $\cM_{0,n}$ and the KZ equation on $\cM_{0,n}$. In Section 5, we discuss the KZ equation on $\cM_{0,4}$. In this case, KZE is regarded as the KZ equation of one variable (KZE1) in the cubic coordinate system of $\cM_{0,4}$. Moreover we consider generalization of KZE1 known as the Schlesinger type equations. In section 6, we consider the harmonic product of MZVs and MPLs. Similarly as in the case of the harmonic product of MZVs, the harmonic product of MPLs is defined as a sort of ``series shuffle product''.

In Section 7, we consider KZE on $\cM_{0,5}$. In the cubic coordinate system of $\cM_{0,5}$, the equation can be written as the KZ equation of two variables (KZE2). We show the decomposition of the infinitesimal pure braid Lie algebra $\fX$ and its universal enveloping algebra $\cU(\fX)$ in two ways (Proposition \ref{prop:decomposition_fX}). In Section 8, we give an interpretation of the decompositions of $\fX$ based on the fiber space structure of the moduli space $\cM_{0,5}$.

In Section 9, we discuss the reduced bar algebra $\cB$ and its decomposition. Let $S$ be a shuffle algebra of 1-forms appeared in KZE2. The reduced bar algebra $\cB$ is the subalgebra of $S$ spanned by elements which satisfies Chen's integrability condition. An iterated integral of an element of $\cB$ depends only on a homotopy class of the integral contour. $\cB$ is interpreted as the 0-th cohomology group of the reduced bar complex of the Orlik-Solomon algebra associated with $\cM_{0,5}$ \cite{K}.

However we consider $\cB$ from a different view point. We observe that $\cB$ is a dual Hopf algebra of $\cU(\fX)$ (Proposition \ref{prop:UX_B_duality}). We denote by $\cB^0$ the subalgebra of $\cB$ spanned by elements which are regular at the origin. Through the decomposition of $\cU(\fX)$, $\cB$ and $\cB^0$ decompose to tensor product of shuffle algebras in two ways (Proposition \ref{prop:decomposition_cB} and \ref{prop:decomposition_cB0}). Such decomposition of $\cB^0$ corresponds to iterated integrations along the specific contours named $C_{1\otimes2}$ and $C_{2\otimes1}$ (Figure \ref{fig:C12_C21}).

In section 10, we discuss hyperlogarithms of the type $\cM_{0,5}$ and MPLs of two variables through the decompositions of $\cB^0$.

In Section 11, the generalized harmonic product relations (GHPRs, for short) are formulated through iterated integration of an element of $\cB^0$ along the contours $C_{1\otimes2}$ and $C_{2\otimes1}$ (Theorem \ref{thm:generalized_harmonic_product_relation}). Furthermore we prove that the GHPRs contain the harmonic product of MPLs (Theorem \ref{thm:ghpl_harmonic_product}).

In Section 12, we construct the fundamental solution normalized at the origin of KZE2 (Proposition \ref{prop:IISol_2KZ}) and show the decomposition of the fundamental solution (Proposition \ref{prop:decomposition}). Finally in Section 13, we calculate the fundamental solution along the contour $C_{1\otimes2}$, $C_{2\otimes1}$ and show that the GHPRs are equivalent to the relations which comes from the decompositions of the fundamental solution (Theorem \ref{thm:GHPR_is_decomposition}).\\

In \cite{OU2}, we will discuss the connection problem of the KZ equation on $\cM_{0,5}$ from the viewpoint of the decomposition theorem.\\

\paragraph{\bf Acknowledgment}

The first author is supported by Waseda University Grant for Special Research Projects No. 2010B-200, 2011B-095. The second author is partially supported by JPSP Grant-in-Aid No. 19540056, 22540035.

\section{Preliminary}\label{sec:preliminary}

For a Lie algebra $\fg$, we denote by $\cU(\fg)$ the universal enveloping algebra of $\fg$ and $\bunit$ the unit of $\cU(\fg)$. $\cU(\fg)$ has a Hopf algebra structure as follows. We define the coproduct $\Delta$, the counit $\varepsilon$ as an algebra morphism and an antipode $\rho$ as an anti-algebra morphism by
\begin{align*}
\Delta(x)&=\bunit\otimes x+x\otimes \bunit,\\
\varepsilon(x)&=0,\\
\rho(x)&=-x \end{align*}
for $x \in \fg$.

If $\fg$ is a graded Lie algebra $\fg = \bigoplus_{s=1}^{\infty} \fg_s,\; [\fg_s,\fg_{s'}] \subset \fg_{s+s'}$, the universal enveloping algebra $\cU(\fg)$ is also graded:
\begin{gather*}
\cU(\fg)=\bigoplus_{s=0}^{\infty}\cU_s(\fg),\\
\cU_s(\fg)\,
\cU_{s'}(\fg)\subset \cU_{s+s'}(\fg),
\end{gather*}
where $\cU_s(\fg)$ is determined by $\fg_s \subset \cU_s(\fg)$. Moreover $\cU(\fg)$ is a graded Hopf algebra, namely
\begin{align*}
\Delta(\cU_s(\fg)) &\subset \sum_{s'+s''=s}\cU_{s'}(\fg)\otimes \cU_{s''}(\fg),\\
\varepsilon(\cU_s(\fg))&=\{0\}
\end{align*}
for $s \ge 1$ and $\rho$ preserves the grading as an anti-algebra morphism. We denote by $\widetilde{\cU}(\fg)$ the completion of $\cU(\fg)$ with respect to this grading.\\

For a set $\fA=\{A_1,\ldots,A_n\}$, we denote by $\C\{\fA\}=\C\{A_1,\ldots,A_n\}$ the free Lie algebra generated by $A_1,\ldots,A_n$ over $\C$. The universal enveloping algebra $\cU(\C\{\fA\})$ is nothing but the non-commutative polynomial algebra $\C\langle \fA\rangle$ generated by $\fA$. $\cU(\C\{\fA\})$ is a graded Hopf algebra with grading given by the length of the words.\\

The shuffle algebra $S(A)=S(a_1,\ldots,a_n)$ generated by the set $A$ of $a_1,\ldots,a_n$ (Reutenauer \cite{R}) is a non-commutative polynomial algebra $\C\langle A \rangle$ generated by $A$ with the shuffle product $\sh$. The shuffle product is defined as
\begin{align*}
w \sh \bnull&=\bnull\sh w=w,\\
(a_1 w_1) \sh (a_2 w_2)&=a_1 (w_1 \sh (a_2 w_2)) + a_2 ((a_1 w_1) \sh w_2),
\end{align*}
where $\bnull$ is the unit of $\C\langle A \rangle$ (that is, $\bnull$ is the empty word), $a_1,a_2 \in A$, and $w,w_1,w_2$ are words of $\C\langle A \rangle$. $(S(A),\sh,\bnull)$ is an associative commutative algebra and has a Hopf algebra structure by the coproduct
\begin{equation*}
\Delta^{*}(a_{i_1}\cdots a_{i_r})=\sum_{k=0}^r a_{i_1}\cdots a_{i_k}\otimes a_{i_{k+1}}\cdots a_r,
\end{equation*}
where $a_{i_k} \in A$ (we regard $a_{i_1}\cdots a_{i_0}$ for $k=0$ and $a_{i_{r+1}}\cdots a_{i_r}$ for $k=r$ as $\bnull$), the counit $\varepsilon^{*}(a)=0$ for $a \in A$ and the antipode $\rho^{*}(a_{i_1}\cdots a_{i_r})=(-1)^r a_{i_r}\cdots a_{i_1}$ for $a_{i_1},\ldots,a_{i_r} \in A$.  $S(A)=\bigoplus_{s=0}^{\infty}S_s(A)$ is also a graded Hopf algebra with the grading defined by the length of words. The restricted dual of this Hopf algebra is the universal enveloping algebra of the free Lie algebra $\cU(\C\{A_1,\ldots,A_n\})$ with respect to the duality defined by the pairing
\begin{equation*}
\langle a_{i_1}\cdots a_{i_r},A_{j_1}\cdots A_{j_s}\rangle=
\begin{cases}
1&(r=s, i_k=j_k \text{ for }1\le k \le r),\\
0&\text{(otherwise)}.
\end{cases}
\end{equation*}

\section{The moduli space $\cM_{0,n}$}
We denote by 
\begin{equation}
\F_n(\bP^1)=\{(x_1,\ldots,x_n) \in (\bP^1)^n \;|\; x_i\neq x_j\;\; (i\neq j)\}
\end{equation}
the configuration space of $n$ points on $\bP^1=\bP^1_{\C}$. The projective general linear group $\PGL(2,\C)$ acts on $\F_n(\bP^1)$ as a linear fractional transformation diagonally. The quotient space by this action
\begin{equation}
\cM_{0,n}=\PGL(2,\C) \backslash \F_n(\bP^1)
\end{equation}
is called the moduli space. There are topological homeomorphisms
\begin{equation*}
\F_n(\bP^1)\approx \PGL(2,\C) \times \cM_{0,n}
\end{equation*}
and
\begin{equation*}
\cM_{0,n}\approx \F_{n-3}(\bP^1-\{0,1,\infty\})
\end{equation*}
for $n\ge 4$. Especially we can identify $\cM_{0,4}$ as $\bP^1-\{0,1,\infty\}$. We define $[x_1,\ldots,x_n]$ to be the $\PGL(2,\C)$ orbit of $(x_1,\ldots,x_n)$. This is called the homogeneous coordinate system of $\cM_{0,5}$.

We introduce two coordinate systems on $\cM_{0,n}$ introduced by Brown \cite{B}. The first one is the simplicial coordinate system
\begin{equation}
y_i=\frac{x_i-x_{n-2}}{x_i-x_n} \cdot \frac{x_{n-1}-x_n}{x_{n-1}-x_{n-2}} \quad (1 \leq i \leq n-3).
\end{equation}
The correspondence
\begin{equation*}
[x_1,\ldots,x_n]\mapsto (y_1,\ldots,y_{n-3})
\end{equation*}
gives the homeomorphism
\begin{equation*}
\cM_{0,n}\approx \F_{n-3}(\bP^1-\{0,1,\infty\}).
\end{equation*}

The second is the cubic coordinate system introduced through
\begin{equation}
z_1=y_1, z_2= \frac{y_2}{y_1}, \cdots, z_{n-3}= \frac{y_{n-3}}{y_{n-4}} \ \Longleftrightarrow \ y_i=z_1\cdots z_i \quad (1 \leq i \leq n-3).
\end{equation}
The cubic coordinate system gives a blowing up for the simplicial coordinate system to be normal crossing at the origin.

 There is a smooth compactification of $\cM_{0,n}$ denoted by $\overline{\cM}_{0,n}$. For $\cM_{0,5}$, the compactification $\overline{\cM}_{0,5}$ (Yoshida \cite{Y}) can be written as
\begin{equation}
\overline{\cM}_{0,5}=\PGL(2,\C) \backslash \{(x_1,\ldots,x_5)\;|\; \text{if $x_i=x_j=x_k$, }\#\{i,j,k\}\le 2\}
\end{equation} 
and 
\begin{equation}
\cM_{0,5}=\overline{\cM}_{0,5}-\bigcup_{1\le i < j \le 5} D_{ij}
\end{equation}
holds, where $D_{ij}=\{[x_1,\ldots,x_5] \in \overline{\cM}_{0,5} \;|\; x_i=x_j\}$. We have
\begin{equation}
D_{ij}\cap D_{kl}=\begin{cases}
\emptyset & (\{i,j\}\cap\{k,l\}\neq \emptyset, \{i,j\}\neq\{k,l\})\\
\{\text{one point}\} & (\{i,j\}\cap\{k,l\}= \emptyset),
\end{cases}
\end{equation}
thus the divisors $\{D_{ij}\}$ are normal crossing. Moreover we obtain
\begin{equation*}
D_{ij}-\bigcup_{\{k,l\}\neq \{i,j\}}D_{kl} \approx \cM_{0,4}.
\end{equation*}
The real part of $\cM_{0,5}$, defined by
\begin{equation}
\cM_{0,5}(\R)=\PGL(2,\R)\backslash \{(x_1,x_2,x_3,x_4,x_5) \in \R^5\;|\;x_i\neq x_j \;(i\neq j)\}
\end{equation}
has 12 distinct connected components and each component is homeomorphic a pentagon like Figure \ref{fig:m05}.\\

\begin{figure}[h]
\begin{picture}(0,4)(-4,0)
\put(0.5,0){\scalebox{0.5}{\includegraphics{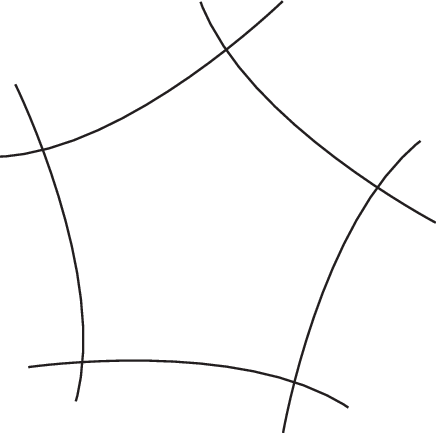}}}
\put(1.7,0.25){\footnotesize $D_{23}$}
\put(0.4,1.5){\footnotesize $D_{45}$}
\put(1.2,2.9){\footnotesize $D_{12}$}
\put(3,2.7){\footnotesize $D_{35}$}
\put(3.4,1.1){\footnotesize $D_{14}$}
\end{picture}
\caption{the homogeneous coordinate system of $\cM_{0,5}$}
\label{fig:m05}
\end{figure}
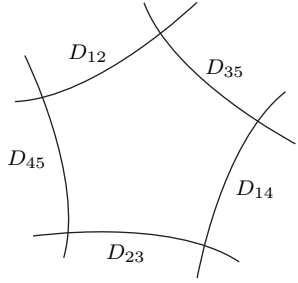

The divisors $\{D_{ij}\}$ are blowing down to the divisors
\begin{equation}
\{y_1=0,1,\infty\}\cup\{y_2=0,1,\infty\}\cup\{y_1=y_2\}
\end{equation}
in the simplicial coordinate system, and
\begin{equation}
\{z_1=0,1,\infty\}\cup\{z_2=0,1,\infty\}\cup\{z_1z_2=1\}
\end{equation}
in the cubic coordinate system. The real part of these divisors are drawn as Figure \ref{fig:m05_simplicial_cubic}.

\begin{figure}[h]
\begin{picture}(0,4)(-2.5,0)
\put(-0.8,0.5){\scalebox{0.33}{\includegraphics{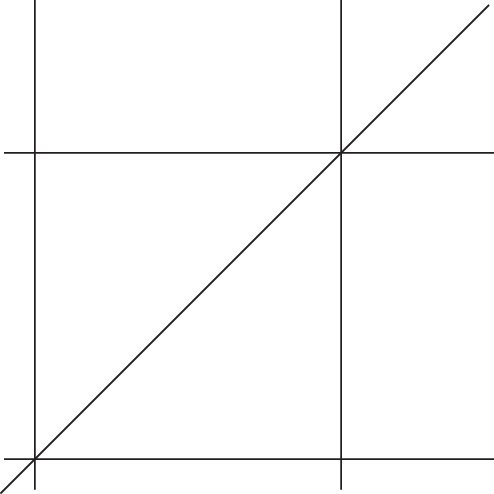}}}
\put(-1.2,0.2){\footnotesize $(0,0)$}
\put(1.1,0.3){\footnotesize $(1,0)$}
\put(-1.5,2.6){\footnotesize $(0,1)$}
\put(1.4,2.5){\footnotesize $(1,1)$}
\put(2.1,0.6){\footnotesize $y_1$}
\put(-0.7,3.4){\footnotesize $y_2$}

\put(5.4,0.4){\scalebox{0.33}{\includegraphics{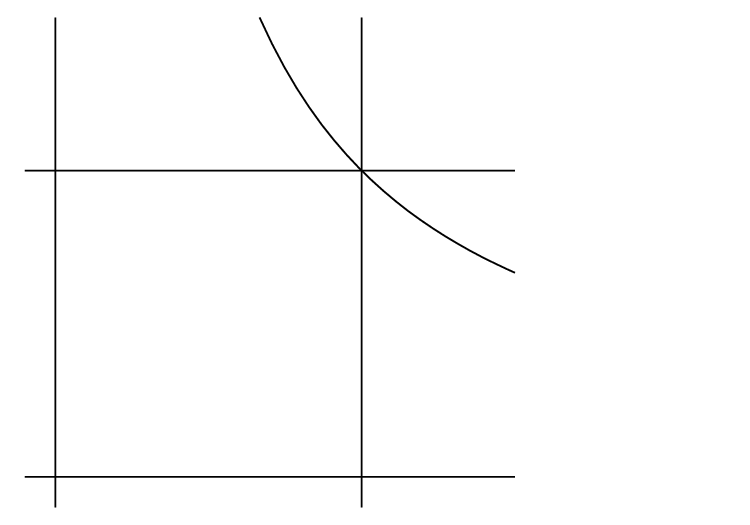}}}
\put(5,0.2){\footnotesize $(0,0)$}
\put(7.3,0.3){\footnotesize $(1,0)$}
\put(4.7,2.6){\footnotesize $(0,1)$}
\put(7.5,2.6){\footnotesize $(1,1)$}
\put(8.3,0.6){\footnotesize $z_1$}
\put(5.5,3.4){\footnotesize $z_2$}
\end{picture}
\caption{the simplicial and cubic coordinate system of $\cM_{0,5}$}
\label{fig:m05_simplicial_cubic}
\end{figure}
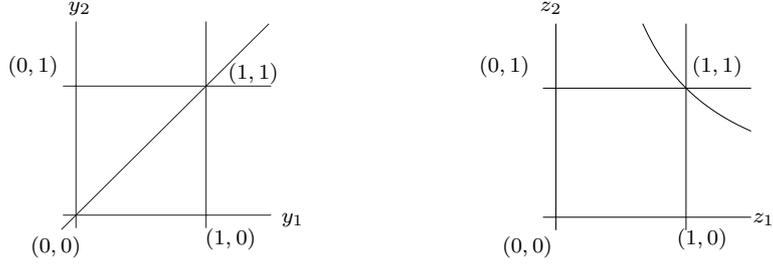

\section{KZ equation on $\cM_{0,n}$}

Denote by $\fX$ the Lie algebra derived by the lower central series of the fundamental group of $\F_n(\bP^1)$(Ihara \cite{I}). The Lie algebra $\fX$ is written as
\begin{equation}
\fX:= \C\{\varOmega_{ij}\;|\;{1\le i,j \le n}\}\Big/ (\text{IPBR}),
\end{equation}
and referred to as the infinitesimal pure braid Lie algebra. The infinitesimal pure braid relation (IPBR) is defined by
\begin{equation}
\begin{cases}
\varOmega_{ij}=\varOmega_{ji}, \quad &\varOmega_{ii}=0,\\
\sum_j \varOmega_{ij}=0 \quad (\forall i ),  \quad &[\varOmega_{ij},\varOmega_{kl}]=0 \quad (\{i,j\} \cap \{k,l\} = \emptyset). \label{IPBR}
\end{cases}
\end{equation}
Since the relations \eqref{IPBR} are homogeneous, the infinitesimal pure braid Lie algebra $\fX$ has the natural grading derived from $\C\{\varOmega_{ij}\}$, and $\cU(\fX)$ has a graded Hopf algebra structure.\\

We call the total differential equation (connection) on $\F_n(\bP^1)$
\begin{equation}\label{KZeq}
dG=\varOmega G,  \qquad \varOmega = \sum_{i<j}\omega_{ij}\varOmega_{ij}, \qquad \omega_{ij}=d\log(x_i-x_j)
\end{equation}
the KZ equation(Drinfel'd \cite{D}). We note that this equation is defined on $\F_n(\C)$, but the relation $\sum_j \varOmega_{ij}=0$ on \eqref{IPBR} makes it regular at infinity. There are unique non-trivial second order relations among $\omega_{ij}$'s called the Arnold relations(Arnold \cite{A})
\begin{equation}\label{AR}
      \omega_{ij} \wedge \omega_{ik} + \omega_{ik} \wedge \omega_{jk} + 
      \omega_{jk} \wedge \omega_{ij} = 0
\end{equation}
for $i<j<j$.

From \eqref{IPBR} and \eqref{AR}, the equation \eqref{KZeq} is an integrable system and has $\PGL(2,\C)$-invariance. Thus one can regard \eqref{KZeq} as an equation on the moduli space $\cM_{0,n}$ and solutions of the equation are considered as $\widetilde{\cU}(\fX)$-valued functions on $\cM_{0,n}$. The equation \eqref{KZeq} has logarithmic singularities along the divisors $D=\{D_{ij}\}$, where $D_{ij}=\{[x_1,\ldots,x_n] \in \overline{\cM}_{0,n}\;|\; x_i=x_j\}$.

\section{KZ equation of one variable}\label{sec:KZ1}

The KZ equation on $\cM_{0,4}$ is described in the cubic coordinate system as
\begin{equation}\label{1KZeq}
dG=\varOmega G, \qquad \varOmega=\xi_1X_1+\xi_{11}X_{11}, \qquad \xi_1=\frac{dz_1}{z_1},\; \xi_{11}=\frac{dz_1}{1-z_1},
\end{equation}
where $X_1=\varOmega_{12},\; X_{11}=-\varOmega_{13}$. We call the equation \eqref{1KZeq} the KZ equation of one variable (KZE1). The equation \eqref{1KZeq} has logarithmic singularities at $D=\{z_1=0,1,\infty\}$. In this case, the infinitesimal pure braid relation \eqref{IPBR} is trivial, thus the infinitesimal pure braid Lie algebra $\fX=\C\{X_1,X_{11}\}$ is a free Lie algebra generated by $X_1,X_{11}$. The Arnold relations \eqref{AR} have only a trivial relation $\xi_1 \wedge \xi_{11}=0$.\\

We consider the fundamental solution $\cL(z_1)$ normalized at the origin of \eqref{1KZeq}. It is the solution expressed as $\cL(z_1)=\hcL(z_1)z_1^{X_1}$, where $\hcL(z_1)$ is a $\widetilde{\cU}(\fX)$-valued holomorphic function on a neighborhood at $z_1=0$ and $\hcL(0)=\bunit$.

\begin{prop}\label{prop:fundamental_solution_KZ1}
The fundamental solution $\cL(z_1)$ normalized at the origin exists uniquely and expressed as follows:
\begin{align}
\hcL(z_1)&=\sum_{s=0}^{\infty} \hcL_s(z_1), \label{1KZSol}\\
\hcL_s(z_1) &=\sum_{k_1+\cdots+k_r=s} \left(\int_0^{z_1} \xi_1^{k_1-1}\xi_{11} \cdots \xi_1^{k_r-1} \xi_{11}\right) \label{1KZSol2} \\[-1ex]
& \hspace{2cm}\times \ad(X_1)^{k_1-1}\mu(X_{11})\cdots\ad(X_1)^{k_r-1}\mu(X_{11})(\bunit),\notag
\end{align}
where $\ad(X_1)(F)=[X_1,F],\;\; \mu(X_{11})(F)=X_{11}F$ for $F \in \cU(\fX)$.
\end{prop}

The integral in the right hand side of \eqref{1KZSol2} stands for the iterated integral; it is defined by
\begin{equation}
\int_{z_1^{(0)}}^{z_1}\xi_i w:=\int_{z_1^{(0)}}^{z_1}\left(\xi_i \int_{z_1^{(0)}}^{z_1}w\right),\qquad \int_{z_1^{(0)}}^{z_1}\bnull:=1,
\end{equation}
where $i=1$ or $11$, and $w$ is a word of the shuffle algebra $S(\xi_1,\xi_{11})$. If $w$ is not terminated by $\xi_{1}$, we can set the start point $z_1^{(0)}$ as 0.

\begin{proof}[Proof of Proposition \ref{prop:fundamental_solution_KZ1}]
The degree $s$ homogeneous component of the holomorphic part $\hcL_s(z_1)$ satisfies the following recursive equation
\begin{equation*}
\frac{d\hcL_{s+1}(z_1)}{dz_1}=\frac{1}{z_1}[X_1,\hcL_s(z_1)]+\frac{1}{1-z_1}X_{11}\hcL_s(z_1).
\end{equation*}
By starting from $\hcL_0(z_1)=\bunit$, we have
\begin{equation*}
\hcL_s(z_1)=\int_0^{z_1}(\xi_1\otimes\ad(X_1)+\xi_{11}\otimes\mu(X_{11}))^s(\bnull\otimes\bunit)
\end{equation*}
inductively. This is nothing but the equation \eqref{1KZSol2}.
\end{proof}

We note that the fundamental solution $\cL(z_1)$ is a grouplike element of $\widetilde{\cU}(\fX)$, namely
\begin{align}
\Delta(\cL(z_1))&=\cL(z_1)\otimes\cL(z_1),\\
\varepsilon(\cL(z_1))&=1.
\end{align}

The function defined by the iterated integral $\int_0^{z_1}$ from $0$ to $z_1$ of a word $\xi_1^{k_1-1} \xi_{11} \cdots \xi_1^{k_r-1} \xi_{11}$ is denoted by $\Li_{k_1,\ldots,k_r}(z_1)$
\begin{equation}\label{1MPL}
\Li_{k_1,\ldots,k_r}(z_1):=\int_0^{z_1} \xi_1^{k_1-1} \xi_{11} \cdots \xi_1^{k_r-1} \xi_{11}.
\end{equation}
We call this a multiple polylogarithm of one variable (MPL1, for short). MPL1 is a many-valued analytic function on $\bP_1-\{0,1,\infty\}$ and has a Taylor expansion at $z=0$ as follows:
\begin{equation}
\Li_{k_1,\ldots,k_r}(z_1)=\sum_{m_1>\cdots>m_r>0}\frac{z_1^{m_1}}{m_1^{k_1}\cdots m_r^{k_r}}. \label{MPL}
\end{equation}
This series converges absolutely on $|z_1|<1$. If $k_1>1$, the series also converges as $z_1$ tends to 1 and defines a multiple zeta value (MZV)
\begin{equation}
\zeta(k_1,\ldots,k_r):=\lim_{z_1 \to 1}\Li_{k_1,\ldots,k_r}(z_1)=\sum_{m_1>\cdots>m_r>0}\frac{1}{m_1^{k_1}\cdots m_r^{k_r}}. \label{MZV}
\end{equation}

The sum of index $k=k_1+\cdots+k_r$ is called the weight and the length of index $r$ is referred to as the depth.\\

For a word $\xi^{k_1-1}\xi_{11}\cdots\xi^{k_r-1}\xi_{11} \in S(\xi_1,\xi_{11})$, we put
\begin{align}
&\Li(\xi_1^{k_1-1} \xi_{11}\cdots \xi_1^{k_r-1} \xi_{11};z_1)=\Li_{k_1,\ldots,k_r}(z_1), \label{Li_notation}\\
&\zeta(\xi_1^{k_1-1} \xi_{11}\cdots \xi_1^{k_r-1} \xi_{11})=\zeta(k_1,\ldots,k_r) \qquad (k_1>1) \label{zeta_notation}
\end{align}
and extend $\Li$ and $\zeta$ to the $\C$ linear map
\begin{align*}
\Li: &\;S^0(\xi_1,\xi_{11})=\C\bnull+S(\xi_1,\xi_{11})\xi_{11} \to \{\C\text{-valued analytic functions}\},\\
\zeta: &\;S^{10}(\xi_1,\xi_{11})=\C\bnull+\xi_1S(\xi_1,\xi_{11})\xi_{11} \to \C.
\end{align*}

Let $a_1,\ldots,a_n \in \C-\{0\}$ be distinct non-zero complex numbers. KZE1 is generalized to the following total differential equation
\begin{equation}\label{g1KZ}
dG=\left(\frac{A_0dz}{z}+\sum_{i=1}^n \frac{a_iA_idz}{1-a_iz}\right)G
\end{equation}
on $\bP^1-\{0,\frac{1}{a_1},\ldots,\frac{1}{a_n},\infty\}$. The coefficients $A_0,A_1,\ldots,A_n$ are assumed to generate the free Lie algebra $\C\{A_0,A_1,\ldots,A_n\}$. We call this equation the generalized KZ equation of one variable or the Schlesinger type equation. The fundamental solution normalized at the origin of the equation can be constructed in a similar fashion as KZE1.

\begin{prop}
Put $\omega_0=\frac{dz}{z},\omega_1=\frac{a_1dz}{1-a_1z},\ldots,\omega_n=\frac{a_ndz}{1-a_nz}$. The fundamental solution $\cL(z)$ normalized at $z_1=0$ of \eqref{g1KZ} exists uniquely and expressed as follows:
\begin{align}
\cL(z)&=\hcL(z)z^{A_0},\\
\hcL(z)&=\sum_{s=0}^{\infty} \hcL_s(z), \label{g1KZSol}\\
\hcL_s(z) &=\sum_{\substack{i_1,\ldots,i_r \in \{1,\ldots,n\}\\k_1+\cdots+k_r=s}} \left(\int_0^{z} \omega_0^{k_1-1}\omega_{i_1} \cdots \omega_0^{k_r-1} \omega_{i_r}\right) \notag \\[-1ex]
& \hspace{2cm}\times \ad(A_0)^{k_1-1}\mu(A_{i_1})\cdots\ad(A_0)^{k_r-1}\mu(A_{i_r})(\bunit).
\end{align}
\end{prop}
We call a function defined by the iterated integral $\int_0^{z} \omega_0^{k_1-1}\omega_{i_1} \cdots \omega_0^{k_r-1} \omega_{i_r}$ a hyperlogarithm and denote it by
\begin{equation}
L({}^{k_1}a_{i_1}\ldots{}^{k_r}a_{i_r};z_1):=\int_0^{z} \omega_0^{k_1-1}\omega_{i_1} \cdots \omega_0^{k_r-1} \omega_{i_r}.
\end{equation}
A hyperlogarithm is a many-valued analytic function on $\bP^1-\{0,\frac{1}{a_1},\ldots,\frac{1}{a_n},\infty\}$ and has a Taylor expansion
\begin{equation}
L({}^{k_1}a_{i_1}\cdots{}^{k_r}a_{i_r};z_1)=\sum_{m_1>m_2>\cdots>m_r>0}\frac{a_{i_1}^{m_1-m_2}a_{i_2}^{m_2-m_3}\cdots a_{i_r}^{m_r}}{m_1^{k_1}\cdots m_r^{k_r}} {z_1}^{m_1}
\end{equation}
on a neighborhood of $z_1=0$. This series converges absolutely on
\begin{equation*}
|z_1|<\min\{1/|a_i|\;|\;1\le i \le n\}.
\end{equation*}

For $n=1$ and $a_1=1$, the generalized KZ equation of one variable is nothing but KZE1 and hyperlogarithms are MPL1s. For $n=2, a_1=1$ and $a_2=z_2$, hyperlogarithms are referred to as hyperlogarithms of the type $\cM_{0,5}$. We will discuss them in Section \ref{sec:hyperlogarithm}.

\section{The harmonic product of MZVs and MPLs}\label{sec:HP}

For the depth 1, MZVs are nothing but the Riemann zeta values. For this case, we have
\begin{align}
\zeta(k_1)\zeta(l_1)&=\sum_{m_1>0}\frac{1}{m_1^{k_1}}\sum_{n_1>0}\frac{1}{n_1^{l_1}}\\
&=\left(\sum_{m_1>n_1>0}+\sum_{(m_1=n_1)>0}+\sum_{n_1>m_1>0}\right)\frac{1}{m_1^{k_1}n_1^{l_1}} \notag \\
&=\zeta(k_1,l_1)+\zeta(k_1+l_1)+\zeta(l_1,k_1). \notag
\end{align}
These relations can be generalized by using Hoffman's harmonic algebra (Hoffman \cite{H}) as follows. The harmonic product $*$ on $S^0=S^0(\xi_1,\xi_{11})=\C\bnull+S(\xi_1,\xi_{11})\xi_{11} \subset S(\xi_1,\xi_{11})$ is defined as
\begin{align}
&w_1 * \bnull=\bnull * w_1 = w_1,\\
&(\chi_{k_1} w_1)*(\chi_{k_2} w_2)\\
&\qquad \qquad =\chi_{k_1}(w_1*(\chi_{k_2} w_2))+\chi_{k_2}((\chi_{k_1} w_1)*w_2)+\chi_{k_1+k_2}(w_1*w_2) \notag
\end{align}
for any words $w_1,w_2 \in S^0$ and $\chi_k$ stands for $\xi_1^{k-1} \xi_{11}$. Then $(S^0,*,\bnull)$ becomes an associative commutative algebra. We remark that the harmonic product is expressed in a recursive form as follows:
\begin{align}
&(\chi_{k_1}\cdots\chi_{k_r})*(\chi_{l_1}\cdots\chi_{l_s}) \label{HP-RD}\\
&\qquad =\sum_{p=1}^{r-1} \Big((\chi_{k_1}\cdots\chi_{k_p}\chi_{l_1})((\chi_{k_{p+1}}\cdots\chi_{k_r})*(\chi_{l_2}\cdots\chi_{l_s})) \notag \\*[-3.5ex]
&\qquad \qquad\qquad\phantom{=} +(\chi_{k_1}\cdots\chi_{k_p}\chi_{k_{p+1}+l_1})((\chi_{k_{p+2}}\cdots\chi_{k_r})*(\chi_{l_2}\cdots\chi_{l_s})) \Big) \notag \\*[-1ex]
&\qquad \qquad \phantom{=} +\chi_{k_1}\cdots\chi_{k_r}\chi_{l_1}\cdots\chi_{l_s} \notag \\*[-1ex]
&\qquad \phantom{=} +\sum_{p=1}^{s-1} \Big((\chi_{l_1}\cdots\chi_{l_p}\chi_{k_1})((\chi_{k_2}\cdots\chi_{k_r})*(\chi_{l_{p+1}}\cdots\chi_{l_s})) \notag \\*[-3.5ex]
&\qquad \qquad\qquad \phantom{=} +(\chi_{l_1}\cdots\chi_{l_p}\chi_{k_1+l_{p+1}})((\chi_{k_2}\cdots\chi_{k_r})*(\chi_{l_{p+2}}\cdots\chi_{l_s})) \Big) \notag \\*[-1ex]
&\qquad \qquad \phantom{=} +\chi_{l_1}\cdots\chi_{l_s}\chi_{k_1}\cdots\chi_{k_r} \notag \\*
&\qquad \phantom{=} +\chi_{k_1+l_1}((\chi_{k_2}\cdots\chi_{k_r})*(\chi_{l_2}\cdots\chi_{l_s})), \notag
\end{align}
where we regard $(\chi_{k_{i+1}}\cdots\chi_{k_i})$ and $(\chi_{l_{i+1}}\cdots\chi_{l_i})$ as $\bnull$. Furthermore, for words $w_1,w_2$ in $S^{10}=S^{10}(\xi_1,\xi_{11})=\C\bnull+\xi_1S(\xi_1,\xi_{11}) \xi_{11}$, we have
\begin{equation}\label{HP-MZV}
\zeta(w_1)\zeta(w_2)=\zeta(w_1 * w_2).
\end{equation}
This is the harmonic product of MZVs.\\

We denote by $\cIDX$ the non-commutative algebra over $\C$ spanned by all indexes of positive integers and the empty index $\emptyset$. The product of index $(k_1,\ldots,k_r)\cdot(l_1,\ldots,l_s)$ is defined by the concatenation $(k_1,\ldots,k_r,l_1,\ldots,l_s)$. There is an algebra isomorphism from $S^0$ to $\cIDX$ defined by $\chi_{k_1}\cdots\chi_{k_r} \mapsto (k_1,\ldots,k_r)$, and $\xi$ (resp. $\Li$) can be regarded as a $\C$ linear map on $\cIDX$. One can express the equation \eqref{HP-RD} and define the harmonic product on $\cIDX$ as follows;
\begin{align}
&(k_1,\ldots,k_r)*\emptyset=\emptyset*(k_1,\ldots,k_r)=(k_1,\ldots,k_r),\\
&(k_1,\ldots,k_r)*(l_1,\ldots,l_s) \label{index-HD-RD}\\
&\qquad=\sum_{p=1}^{r-1} \Big((k_1,\ldots,k_p,l_1)\cdot((k_{p+1},\ldots,k_r)*(l_2,\ldots,l_s)) \notag \\*[-3.5ex]
&\qquad\qquad\qquad\phantom{=} +(k_1,\ldots,k_p,k_{p+1}+l_1)\cdot((k_{p+2},\ldots,k_r)*(l_2,\ldots,l_s)) \Big) \notag \\*[-1ex]
&\qquad\qquad\phantom{=} +(k_1,\ldots,k_r,l_1,\ldots,l_s) \notag \\*[-1ex]
&\qquad\phantom{=} +\sum_{p=1}^{s-1} \Big((l_1,\ldots,l_p,k_1)\cdot((k_2,\ldots,k_r)*(l_{p+1},\ldots,l_s)) \notag \\*[-3.5ex]
&\qquad\qquad\qquad\phantom{=} +(l_1,\ldots,l_p,k_1+l_{p+1})\cdot((k_2,\ldots,k_r)*(l_{p+2},\ldots,l_s)) \Big) \notag \\*[-1ex]
&\qquad\qquad\phantom{=} +(l_1,\ldots,l_s,k_1,\ldots,k_r) \notag \\*
&\qquad\phantom{=} +(k_1+l_1)\cdot((k_2,\ldots,k_r)*(l_2,\ldots,l_s)), \notag
\end{align}
where $(k_{r+1},\ldots,k_{r})=(l_{s+1},\ldots,l_s)=\emptyset$. In this notation, we can describe the harmonic product of MZVs as
\begin{equation}
\zeta(k_1,\ldots,k_r)*\zeta(l_1,\ldots,l_s)=\zeta((k_1,\ldots,k_r)*(l_1,\ldots,l_s))\end{equation}
for $k_1,l_1>1$.\\

This product is generalized to multiple polylogarithms as follows. Set 
\begin{equation}\label{2MPL}
\Li_{k_1,\ldots,k_r}(i,r-i;z_1,z_2):=\sum_{m_1>\cdots>m_r>0}\frac{z_1^{m_1}z_2^{m_{i+1}}}{m_1^{k_1}\cdots m_r^{k_r}},
\end{equation}
and call this a multiple polylogarithm of two variables (MPL2). This is a special case of hyperlogarithms (for detail, see Section \ref{sec:hyperlogarithm}) and in particular we have
\begin{align}
\Li_{k_1,\ldots,k_r}(r,0;z_1,z_2)&=\Li_{k_1,\ldots,k_r}(z_1),\\
\Li_{k_1,\ldots,k_r}(0,r;z_1,z_2)&=\Li_{k_1,\ldots,k_r}(z_1z_2).
\end{align}
For MPL1s of depth 1, we have
\begin{align}
\Li_{k_1}(z_1)&\Li_{l_1}(z_2)=\sum_{m_1>0}\frac{z_1^{m_1}}{m_1^{k_1}}\sum_{n_1>0}\frac{z_2^{n_1}}{n_1^{l_1}} \label{MPL-HP1}\\
&=\left(\sum_{m_1>n_1>0}\!+\!\sum_{(m_1=n_1)>0}\!+\!\sum_{n_1>m_1>0}\right)\frac{z_1^{m_1}z_2^{n_1}}{m_1^{k_1}n_1^{l_1}} \notag\\
&=\Li_{k_1,l_1}(1,1;z_1,z_2)+\Li_{k_1+l_1}(z_1z_2)+\Li_{l_1,k_1}(1,1;z_2,z_1). \notag
\end{align}
In the same fashion, the product $\Li_{k_1,\ldots,k_r}(z_1)\Li_{l_1,\ldots,l_s}(z_2)$ is expressed as follows. We denote the harmonic product of indexes by
\begin{equation*}
(k_1,\ldots,k_r)*(l_1,\ldots,l_s)=\sum (p_1,\ldots,p_t).
\end{equation*}
By definition, $p_i$ ($i=1,\ldots,t$) is $k_u$, $l_v$ or $k_u+l_v$ for some $u \in \{1,\ldots,r\},\; v \in \{1,\ldots,s\}$ and each of $k_1,\ldots,k_r,l_1,\ldots,l_s$ appears just one time in $(p_1,\ldots,p_t)$. Under this notation, the harmonic product of MPLs is given by
\begin{equation}\label{MPL-HP}
\Li_{k_1,\ldots,k_r}(z_1)\Li_{l_1,\ldots,l_s}(z_2)=\sum_{(p_1,\ldots,p_t)}\sum_{m_1>\cdots>m_p>0}\frac{z_1^{m_{i_1}}z_2^{m_{i_2}}}{m_1^{p_1}\cdots m_t^{p_t}},
\end{equation}
where $k_1$ appears in $p_{i_1}$ and $l_1$ appears in $p_{i_2}$. We note that each summand of the right hand side is
\begin{equation*}
\sum_{m_1>\cdots>m_p>0}\frac{z_1^{m_{i_1}}z_2^{m_{i_2}}}{m_1^{p_1}\cdots m_t^{p_t}}=
\begin{cases}
\Li_{p_1,\ldots,p_t}(i_2-1,t-i_2+1;z_1,z_2) & (i_1=1<i_2)\\
\Li_{p_1,\ldots,p_t}(i_1-1,t-i_1+1;z_2,z_1) & (i_2=1<i_1)\\
\Li_{p_1,\ldots,p_t}(0,t;z_1z_2) & (i_1=i_2=1).
\end{cases}
\end{equation*}

From \eqref{index-HD-RD}, \eqref{MPL-HP1} and \eqref{MPL-HP}, we see that the harmonic product of MPL1s is given in a recursive way as follows;
\begin{align}
&\Li_{k_1,\ldots,k_r}(z_1)\Li_{l_1,\ldots,l_s}(z_2) \label{MPL-HP-RD}\\
&\quad =\sum_{p=1}^{r-1} \Big(\Li_{(k_1,\ldots,k_p,l_1)\cdot((k_{p+1},\ldots,k_r)*(l_2,\ldots,l_s))}(p,\bullet;z_1,z_2) \notag \\*[-3ex]
&\quad\phantom{=}\qquad\qquad +\Li_{(k_1,\ldots,k_p,k_{p+1}+l_1)\cdot((k_{p+2},\ldots,k_r)*(l_2,\ldots,l_s))}(p,\bullet;z_1,z_2) \Big) \notag \\*
&\quad\phantom{=}\qquad +\Li_{(k_1,\ldots,k_r,l_1,\ldots,l_s)}(r,s;z_1,z_2) \notag \\*
&\quad\phantom{=}+\sum_{p=1}^{s-1} \Big(\Li_{(l_1,\ldots,l_p,k_1)\cdot((k_2,\ldots,k_r)*(l_{p+1},\ldots,l_s))}(p,\bullet;z_2,z_1) \notag \\*[-3ex]
&\quad\phantom{=}\qquad\qquad +\Li_{(l_1,\ldots,l_p,k_1+l_{p+1})\cdot((k_2,\ldots,k_r)*(l_{p+2},\ldots,l_s))}(p,\bullet;z_2,z_1) \Big) \notag \\*
&\quad\phantom{=}\qquad +\Li_{(l_1,\ldots,l_s,k_1,\ldots,k_r)}(s,r;z_2,z_1) \notag \\*
&\quad\phantom{=}+\Li_{(k_1+l_1)\cdot((k_2,\ldots,k_r)*(l_2,\ldots,l_s))}(0,\bullet;z_2,z_1), \notag
\end{align}
where, in $\Li_{k_1,\ldots,k_r}(p,\bullet;z_i,z_j)$, $\bullet=(\text{the depth }r- \text{ the first number }p)$.

\section{KZ equation of two variables and the decomposition of the infinitesimal pure braid Lie algebra}\label{sec:decomposition}

We consider the KZ equation on $\cM_{0,5}$. Put
\begin{equation*}
X_1=\varOmega_{12}+\varOmega_{13}+\varOmega_{23},\; X_{11}=-\varOmega_{14},\; X_2=\varOmega_{23},\; X_{22}=-\varOmega_{12},\ X_{12}=-\varOmega_{24}.
\end{equation*}
Then the KZ equation \eqref{KZeq} on $\cM_{0,5}$ becomes
\begin{gather}\label{2KZeq}
dG = \varOmega G, \qquad \varOmega=\xi_1 X_1 + \xi_{11} X_{11} + \xi_2 X_2 + \xi_{22} X_2 + \xi_{12} X_{12}, \\[1ex]
\xi_1=\frac{dz_1}{z_1},\;\; \xi_{11}=\frac{dz_1}{1-z_1},\;\; \xi_2=\frac{dz_2}{z_2},\;\; \xi_{22}=\frac{dz_2}{1-z_2},\;\; \xi_{12}=\frac{d(z_1z_2)}{1-z_1z_2} \notag
\end{gather}
on the cubic coordinate system. We call this the KZ equation of two variables (KZE2). This equation has logarithmic singularities along the divisors $D=\{z_1=0,1,\infty\}\cup\{z_2=0,1,\infty\}\cup\{z_1z_2=1\}$.

The infinitesimal pure braid relation \eqref{IPBR} reads
\begin{equation}\label{IPBR2}
\begin{cases}
[X_1,X_2]=[X_{11},X_2]=[X_1,X_{22}]=0, \\
[X_{11},X_{22}]=[-X_{11},X_{12}]=[X_{22},X_{12}]=[-X_1+X_2,X_{12}],
\end{cases}
\end{equation}
and the Arnold relations \eqref{AR}
\begin{equation}
\begin{cases}
\xi_1\wedge\xi_{11}=0,\quad \xi_2\wedge\xi_{22}=0,\\
(\xi_1+\xi_2)\wedge\xi_{12}=0,\\
\xi_{11}\wedge\xi_{12}+\xi_{22}\wedge(\xi_{11}-\xi_{12})-\xi_2\wedge\xi_{12}=0.
\end{cases}
\end{equation}
The infinitesimal pure braid Lie algebra $\fX=\C\{\fA\} \Big/ \eqref{IPBR2}$, where $\fA$ is a set of $X_1,X_{11},X_2,X_{22},X_{12}$, has the following decomposition.

\begin{prop}\label{prop:decomposition_fX}
The following decompositions hold as $\C$-vector spaces:
\begin{equation}\label{DX}
\fX= \C\{X_1,X_{11},X_{12}\}\oplus \C\{X_2,X_{22}\} = \C\{X_2,X_{22},X_{12}\}\oplus \C\{X_1,X_{11}\},
\end{equation}
where $\C\{X_1,X_{11},X_{12}\}$ and $\C\{X_2,X_{22},X_{12}\}$ are {\rm Lie} ideals of $\fX$, and
\begin{align}
\cU(\fX)&=\cU(\C\{X_1,X_{11},X_{12}\})\otimes \cU(\C\{X_2,X_{22}\}) \label{DU} \\
& = \cU(\C\{X_2,X_{22},X_{12}\})\otimes \cU(\C\{X_1,X_{11}\}). \notag
\end{align}
\end{prop}

The proposition can be proved by using induction on the degree of the gradation on $\fX$. However we can interpret the decomposition \eqref{DX} of $\fX$ from the geometric viewpoint of $\cM_{0,5}$ as in the following section.

\section{Geometrical interpretation of the decomposition of the infinitesimal pure braid Lie algebra}\label{sec:geometry}

We consider the projection
\begin{equation*}
p_2: \cM_{0,5} \to \cM_{0,4};\; [x_1,x_2,x_3,x_4,x_5] \mapsto [x_1,x_3,x_4,x_5]
\end{equation*}
defined by the homogeneous coordinate system. In the cubic coordinate system, this projection can be identified as $(z_1,z_2) \mapsto z_1$. The projection gives the fiber space structure of $\cM_{0,5}$ over the base space $\cM_{0,4}$ and the fiber of $z_1$ is $\bP^1-\{0,1,\infty,z_1^{-1}\}$ (see Figure \ref{fig:fiber_m05}).

\begin{figure}[h]
\begin{picture}(0,4.2)(-3,-1)
\put(0,-0.9){$\cM_{0,4}\approx \bP^1-\{0,1,\infty\}$}
\put(0,0){\line(1,0){5}}
\put(0.5,0){\circle{0.1}}
\put(0.4,-0.4){0}
\put(2.0,0){\circle{0.1}}
\put(1.9,-0.4){1}
\put(3.5,0){\circle{0.1}}
\put(3.3,-0.4){$\infty$}
\put(1.0,2.5){$\cM_{0,5}$}
\put(1.6,2.3){\vector(0,-1){1}}
\put(1.8,1.7){$p_2$}
\put(2.65,0){\circle*{0.1}}
\put(2.55,-0.4){$z_1$}
\put(2.65,0.3){\line(0,1){2.5}}
\put(2.75,1.7){\shortstack[l]{$p_2^{-1}(z_1)$\\\quad $\approx \bP^1-\{0,1,\infty,z_1^{-1}\}$}}
\end{picture}
\caption{the fiber space structure of $\cM_{0,5}$ on $\cM_{0,4}$}
\label{fig:fiber_m05}
\end{figure}
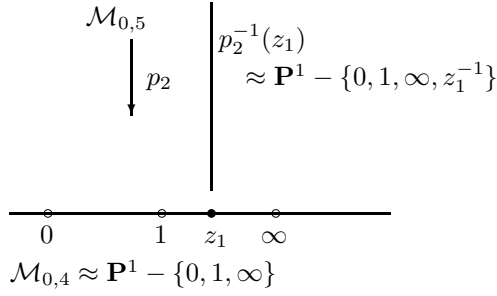

For this fiber space structure, we have the homotopy exact sequence
\begin{multline}\label{homotopy_secuence}
\pi_2(\cM_{0,4},z_1)\;=1 \to \pi_1(\bP^1-\{0,1,\infty,z_1^{-1}\},z_2) \to \pi_1(\cM_{0,5},(z_1,z_2))\\
 \to \pi_1(\cM_{0,4},z_1) \to \pi_0(\bP^1-\{0,1,\infty,z_1^{-1}\},z_2)\;=1.
\end{multline}
There exists a continuous section of this fiber space, so that the exact sequence \eqref{homotopy_secuence} splits. Hence we have the decomposition of the fundamental group
\begin{equation}
\pi_1(\cM_{0,5},(z_1,z_2)) \cong  \pi_1(\bP^1-\{0,1,\infty,z_1^{-1}\},z_2) \rtimes \pi_1(\cM_{0,4},z_1).
\end{equation}
Taking the lower central series of the decomposition (infinitesimal version of the decomposition, Ihara \cite{I} Lemma 3.1.1 and Proposition 3.2.1), we obtain the decomposition of $\fX$
\begin{equation*}
\fX= \C\{X_2,X_{22},X_{12}\}\oplus \C\{X_1,X_{11}\},
\end{equation*}
where $\C\{X_2,X_{22},X_{12}\}$ is a Lie ideal of $\fX$.\\

In the same way, the projection
\begin{equation*}
p_4: \cM_{0,5}\to \cM_{0,4};\; [x_1,x_2,x_3,x_4,x_5] \mapsto [x_1,x_2,x_3,x_5]
\end{equation*}
(in the cubic coordinate system, $(z_1,z_2)\mapsto z_2$) also define the fiber space structure of $\cM_{0,5}$ over $\cM_{0,4}$, and the fiber of $z_2$ is $\bP^1-\{0,1,\infty,z_2^{-1}\}$. Hence we have the decomposition of the fundamental group
\begin{equation}
\pi_1(\cM_{0,5},(z_1,z_2)) \cong  \pi_1(\bP^1-\{0,1,\infty,z_2^{-1}\},z_1)\rtimes \pi_1(\cM_{0,4},z_2)
\end{equation}
and the infinitesimal version
\begin{equation*}
\fX= \C\{X_1,X_{11},X_{12}\}\oplus \C\{X_2,X_{22}\}
\end{equation*}
($\C\{X_1,X_{11},X_{12}\}$ is a Lie ideal) is followed.\\

\section{The reduced bar algebra and iterated integrals of two variables}\label{sec:reduced_bar}

We consider the dual Hopf algebra of $\cU(\fX)$ which consists of differential forms.

Let $A=\{\xi_1,\xi_{11},\xi_2,\xi_{22},\xi_{12}\}$ be a set of differential forms and $S=S(A)$ be the shuffle algebra generated by $A$. We say that a degree $s$ element of $S$
\begin{equation}\label{element_S_s}
S_s(A) \ni \varphi=\sum_{I=(i_1,\ldots,i_s)} \!\!\!\! c_I\; \omega_{i_1}\cdots\omega_{i_s}, \quad (C_I \in \C,\;\; \omega_{i} \in A)
\end{equation}
satisfies Chen's integrability condition, if and only if the formula 
\begin{equation}\label{Chen's_integrability}
\sum_I c_I \; \omega_{i_1}\otimes\cdots\otimes\omega_{i_l}\wedge\omega_{i_{l+1}}\otimes\cdots\otimes\omega_{i_s}=0
\end{equation}
holds for all $1\le l \le s-1$ as a multiple differential form.

From Chen's lemma(Chen \cite{C2}), if an element $\varphi$ \eqref{element_S_s} satisfies Chen's integrability condition, the iterated integral $\int_{(z_1^{(0)},z_2^{(0)})}^{(z_1,z_2)}\varphi$ depends only on the homotopy class of the integral path and defines a many-valued analytic function on $\bP^1 \times \bP^1 - D$.\\

We define the reduced bar algebra $\cB$ as the subspace of $S$ spanned by elements which satisfy Chen's integrability condition. The reduced bar algebra has the grading $\cB=\bigoplus_{s=0}^{\infty} \cB_s,\; \cB_s=\cB\cap S_s(A)$ and $(\cB,\sh,\bnull,\Delta^{*},\varepsilon^{*},\rho^{*})$ is a graded Hopf algebra. The subspaces $\cB_0,\cB_1$ and $\cB_2$ are described as
\begin{align}
\cB_0&=\C \bnull, \label{cB0}\\
\cB_1&=\C \xi_1\oplus\C \xi_{11}\oplus\C \xi_2\oplus\C \xi_{22}\oplus\C \xi_{12}, \label{cB1}\\
\cB_2&=\bigoplus_{\omega \in A}\C \omega\omega \oplus \bigoplus_{i=1,2}\C \xi_i\xi_{ii} \oplus \bigoplus_{i=1,2}\C \xi_{ii}\xi_i \label{cB2}\\*
&\qquad \oplus \bigoplus_{\substack{\omega_1=\xi_1,\xi_{11}\\ \omega_2=\xi_2,\xi_{22}}}\C (\omega_1\omega_2+\omega_2\omega_1) \oplus \bigoplus_{\omega \in A-\{\xi_{12}\}}\C (\omega\xi_{12}+\xi_{12}\omega)\notag \\*
&\qquad\qquad\oplus \C (\xi_1\xi_{12}+\xi_2\xi_{12}) \oplus 
   \C (\xi_{11}\xi_{12}+\xi_{22}\xi_{11}-\xi_{22}\xi_{12}-\xi_2\xi_{12}). \notag
\intertext{Moreover, $\cB_s$ ($s>2$) is characterized as follows(Brown \cite{B}):}
\cB_s&=\bigcap_{j=1}^{s-1}\cB_j\cB_{s-j}=\bigcap_{j=0}^{s-2}\underbrace{\cB_1\cdots\cB_1}_{j\text{ times}}\cB_2\underbrace{\cB_1\cdots\cB_1}_{s-j-2\text{ times}}. \label{cBs_characterize}
\end{align}

\begin{prop}\label{prop:UX_B_duality}
$\cU(\fX)$ is a graded restricted dual Hopf algebra of $\cB$.
\end{prop}

For the proof we need the following notation and lemma. We denote by $\fH=\cU(\C\{\fA\})$ the universal enveloping algebra of the free Lie algebra generated by $\fA=\{X_1,X_{11},X_2,X_{22},X_{12}\}$, and by $\fH=\bigoplus_{s=0}^{\infty}\fH_s$ the grading of $\fH$ defined by the length of words.

Let
\begin{align*}
\fI_2&=\C [X_1, X_2]+\C [X_1, X_{22}]+\C [X_{11}, X_2]+\C([X_{11},X_{22}]+[X_{11},X_{12}])\\
&\hspace{1.5cm} +\C([X_{11},X_{22}]+[X_{12},X_{22}])+\C([X_{11},X_{22}]+[X_1-X_2,X_{12}])
\end{align*}
be a subspace of $\C\{\fA\}$ and $\fI$ (resp. $\fJ$) be an ideal of $\C\{\fA\}$ (resp. a two sided ideal of $\fH$) generated by $\fI_2$. Clearly we have $\fX=\C\{\fA\}/\fI$ and $\cU(\fX)=\fH/\fJ$.

\begin{lem}
\begin{equation*}
(\underbrace{\cB_1\cdots\cB_1}_{l-1\text{ times}}
\cB_2\underbrace{\cB_1\cdots\cB_1}_{s-l-1\text{ times}})^{\bot}\cap\fH_s=\fH_{l-1}\fJ_2\fH_{s-l-1}.
\end{equation*}
\end{lem}

\begin{proof}
Let
\begin{equation*}
\{\xi_{i_1}\cdots\xi_{i_{l-1}}\varphi\xi_{i_{l+2}}\cdots\xi_{i_{s}}\;|\; i_k\in\{1,11,2,22,12\},\; \varphi \in \text{the basis of }\cB_2\}
\end{equation*}
be a basis of $\cB_1\cdots\cB_1\cB_2\cB_1\cdots\cB_1$. 
An element of $\fH_s$ which is orthogonal to the element $\xi_{i_1}\cdots\xi_{i_{l-1}}\varphi\xi_{i_{l+2}}\cdots\xi_{i_{s}}$ can be expressed as $X_{i_1}\cdots X_{i_{l-1}}\Phi X_{i_{l+2}}\cdots X_{i_{s}}$, for some $\Phi \in \fH_2$ such that $\langle\Phi,\varphi\rangle=0$. By counting the dimensions, we obtain $\cB_2^{\bot} \cap \fH_2=\fJ_2$. Thus we have proved the lemma.
\end{proof}

\begin{proof}[Proof of Proposition \ref{prop:UX_B_duality}]
Since  $\fJ_2=\fI_2$ and 
\begin{equation*}
\fJ_s=\sum_{l=1}^{s-1} \fH_{l-1}\fJ_2\fH_{s-l-1},
\end{equation*}
for $s>2$, the lemma says that $\fJ$ is the orthogonal complement of $\cB$.
Hence $\cU(\fX)=\fH/\fJ$ is a dual Hopf algebra of $\cB$.
\end{proof}

\vspace{\baselineskip}

We denote by $\cB^0$ (resp. $S^0$) the subalgebra of $\cB$ (resp. $S$) generated by elements which have no terms ended with $\xi_1$ and $\xi_2$. For an element $\varphi \in \cB^0$, the iterated integral $\int_{(0,0)}^{(z_1,z_2)}\varphi$ makes sense.\\

Corresponding to the decomposition of $\cU(\fX)$, one can show that $\cB$ is isomorphic to the tensor product of shuffle algebras as follows.

\begin{prop}\label{prop:decomposition_cB}
Put $\xi_{12}^{(1)}=\frac{z_2dz_1}{1-z_1z_2}$ and $\xi_{12}^{(2)}=\frac{z_1dz_2}{1-z_1z_2}$. We define the projection $\Pr_{i\otimes j}^{(i)},\Pr_{i\otimes j}^{(j)} \quad (\{i,j\}=\{1,2\})$ by
\begin{align}
\Pr_{i\otimes j}^{(i)}&: S \to S(\xi_i,\xi_{ii},\xi_{12}^{(i)});\quad \xi_j,\xi_{jj}\mapsto 0,\; \xi_{12}\mapsto \xi_{12}^{(i)},\\
\Pr_{i\otimes j}^{(j)}&: S \to S(\xi_j,\xi_{jj});\quad \xi_i,\xi_{ii},\xi_{12} \mapsto 0.
\end{align}
Define maps $\iota_{i\otimes j}:\cB \! \to S(\xi_i,\xi_{ii},\xi_{12}^{(i)})\otimes S(\xi_j,\xi_{jj})$ by
\begin{equation}
\iota_{i\otimes j}=\left(\left.\Pr_{i\otimes j}^{(i)}\right|_{\cB} \otimes \left.\Pr_{i\otimes j}^{(j)}\right|_{\cB}\right)\circ \Delta^{*},
\end{equation}
where $\Delta^{*}$ is the coproduct of $\cB$.
Then $\iota_{i\otimes j}$ are shuffle algebra isomorphisms.
\end{prop}

Furthermore we have a decomposition of $\cB^0$ as follows.
\begin{prop}\label{prop:decomposition_cB0}
The restrictions
\begin{equation*}
\iota_{i\otimes j}|_{\cB^0}: \cB^0 \to S^0(\xi_{i},\xi_{ii},\xi_{12}^{(i)})\otimes S^0(\xi_{j},\xi_{jj})
\end{equation*}
$(\{i,j\}=\{1,2\})$ are also shuffle algebra isomorphisms. 
\end{prop}

We prove only for $\iota_{1\otimes2}$ and $\iota_{1\otimes2}|_{\cB^0}$. We prepare combinatorial claims for $\cB$ and $\cB^0$.

\begin{lem}\label{lem:iota_isomorphism:B:prep}
\begin{enumerate}
\item Assume that $\varphi \in \cB_s \cap \left(S_1(\xi_2,\xi_{22}) S_{s-1}(A)\right)$. Then
\begin{equation*}
\varphi \in S_s(\xi_2,\xi_{22}).
\end{equation*} \label{lem:iota_isomorphism:B:prep:1}
\item For $\varphi \in \cB_s$, the following two statements are equivalent: \label{lem:iota_isomorphism:B:prep:2}
\begin{enumerate}
\item $\varphi$ contains a term of the form $\varphi' \xi_1  \nu_1  \cdots  \nu_p$ $($resp. $\varphi' \xi_2  \mu_1  \cdots  \mu_p$$)$, \label{lem:iota_isomorphism:B:prep:2:1}
\item $\varphi$ contains a term of the form $\varphi'  \nu_1  \cdots  \nu_p  \xi_1$ $($resp. $\varphi'  \mu_1  \cdots  \mu_p  \xi_2$$)$, \label{lem:iota_isomorphism:B:prep:2:2}
\end{enumerate}
where $\varphi' \in \cB$ and $\nu_1,\ldots,\nu_p \in \{\xi_2,\xi_{22}\}$ $($resp. $\mu_1,\ldots,\mu_p \in \{\xi_1,\xi_{11}\}$$)$.
\item Assume that Proposition \ref{prop:decomposition_cB} holds. For $i=1$ and $2$, if $\varphi \in \cB_s$ contains a term which is ended with $\xi_i$, $\iota_{1\otimes2}(\varphi)$ $($resp. $\iota_{2\otimes1}(\varphi)$$)$ also has a term ended with $\xi_i$ in $S(\xi_1,\xi_{11},\xi^{(1)}_{12})$ or $S(\xi_2,\xi_{22})$-component $($resp. $S(\xi_2,\xi_{22},\xi^{(2)}_{12})$ or $S(\xi_1,\xi_{11})$-component $)$. \label{lem:iota_isomorphism:B:prep:3}
\end{enumerate}
\end{lem}

\begin{proof}
\begin{enumerate}
\item This claim can be proved easily by direct calculation for $\cB_2$ and by induction on $s$.

\item Assume that $\varphi \in \cB_s$ contains the term $\varphi' \xi_1  \nu_1  \cdots  \nu_p$. From \eqref{cB2}, $\xi_1 \xi_2$ (resp. $\xi_1  \xi_{22}$) appears in pairs with $\xi_2  \xi_1$ (resp. $\xi_{22}  \xi_1$) in $\cB_2$. Since $\varphi$ belongs to $\cB_{s-p-1}\cB_2\underbrace{\cB_1\cdots\cB_1}_{p \text{ times}}$, $\varphi$ has the term $\varphi' \nu_1  \xi_1 \nu_2 \cdots  \nu_p$. This means that $\varphi$ contains $\varphi' \nu_1  \cdots  \nu_p  \xi_1$. The converse assertion is proved in the same way.

\item We show the claim for $\iota_{1\otimes2}$. Put $\varphi \in \cB_s$ as
\begin{equation*}
\varphi=\varphi_1\xi_1+\varphi_{11}\xi_{11}+\varphi_{12}\xi_{12}+\varphi_2\xi_2+\varphi_{22}\xi_{22},
\end{equation*}
where $\varphi_1,\ldots,\varphi_{22} \in \cB_{s-1}$.

For $i=1$, we have $\varphi_1 \neq 0$ by assumption. Since $\iota_{1\otimes2}$ is a $\sh$-isomorphism, $\iota_{1\otimes2}(\varphi_1)$ is not equal to 0. Thus $\varphi_1$ has a term such as
\begin{equation*}
\mu_1\cdots\mu_p \nu_1\cdots\nu_{s-p}
\end{equation*}
where $\mu_1,\ldots,\mu_p \in \{\xi_1,\xi_{11},\xi_{12}\}$ and $\nu_1,\ldots,\nu_{s-p} \in \{\xi_2,\xi_{22}\}$. This implies that $\varphi$ has the term
\begin{equation*}
\mu_1\cdots\mu_p \nu_1\cdots\nu_{s-p} \xi_1,
\end{equation*}
so has the term
\begin{equation*}
\mu_1\cdots\mu_p \xi_1 \nu_1\cdots\nu_{s-p}
\end{equation*}
by using \eqref{lem:iota_isomorphism:B:prep:1}. Hence $\iota_{1\otimes2}(\varphi)$ contains a term ended with $\xi_1$ in $S(\xi_1,\xi_{11},\xi^{(1)}_{12})$ component.

For $i=2$, we assume that $\varphi_2 \neq 0$. By definition of $\iota_{1\otimes2}$, we have
\begin{align*}
\iota_{1\otimes2}(\varphi)&=(\Pr^{(1)}_{1\otimes2}(\varphi_1)\xi_1)\otimes\bnull
+(\Pr^{(1)}_{1\otimes2}(\varphi_{11})\xi_{11})\otimes\bnull\\*
&\phantom{=}+(\Pr^{(1)}_{1\otimes2}(\varphi_{12})\xi_{12})\otimes\bnull\\*
&\phantom{=}+\iota_{1\otimes2}(\varphi_2) (\bnull\otimes \xi_2)
+\iota_{1\otimes2}(\varphi_{22}) (\bnull\otimes \xi_{22}).
\end{align*}
Thus by virtue of $\iota_{1\otimes2}(\varphi_2) \neq 0$, we see that $\iota_{1\otimes2}(\varphi)$ contains a term ended with $\xi_2$.

\end{enumerate}
\end{proof}

\begin{proof}[Proof of Proposition \ref{prop:decomposition_cB}]
It is enough to prove that $\iota_s:=\iota_{1\otimes2}|_{\cB_s}$ is injective. We show $\ker(\iota_s)=0$ by induction on $s$. For $s=0$ and 1, $\ker(\iota_s)=0$ is clear. For $s=2$, we can show by direct calculation. 
For $s\ge 3$, we assume that $\ker(\iota_{s-1})=0$.

We define the sets
\begin{align*}
S_{1\otimes2,s}(A)&:=\{\varphi\!\in\! S_s(A)\;|\; \text{In all terms of $\varphi$},\\
& \phantom{\{\varphi\!\in\! S_s(A)\;|\; }\text{$\xi_2$ or $\xi_{22}$ appears in the left side of some $\xi_1,\xi_{11},\xi_{12}$}\},\\
S^c_s(\xi_2,\xi_{22})&:=\{\varphi \in S_s(A)\;|\;\text{$\varphi$ is a linear combination of words}\\
&\phantom{=\{\varphi \in S_s(A)\;|\;}\qquad \text{ which have at least one $\xi_1,\xi_{11},\xi_{12}$}\}.
\end{align*}
Then we have clearly $\ker(\iota_s)=\cB_s \cap S_{1\otimes2,s}(A)$.

Put $\varphi \in \ker(\iota_s)$. Since $\varphi \in S_{1\otimes2,s}(A)$, $\varphi$ can be written as
\begin{equation*}
\varphi=\xi_1  \varphi_1 + \xi_{11}  \varphi_{11} + \xi_{12}  \varphi_{12} + \xi_2  \varphi_2 + \xi_{22}  \varphi_{22},
\end{equation*}
where $\varphi_1,\varphi_{11},\varphi_{12}\in S_{1\otimes2,s-1}(A),\;\; \varphi_2,\varphi_{22} \in S^c_{s-1}(\xi_2,\xi_{22})$. On the other hand, from the assumption, we have $\varphi_1,\varphi_{11},\varphi_{12} \in \cB_{s-1}\cap S_{1\otimes2,s-1}(A) = \ker(\iota_{s-1})=0$ and
\begin{equation*}
\varphi=\xi_2  \varphi_2 + \xi_{22}  \varphi_{22}.
\end{equation*}
Thus, from the lemma \eqref{lem:iota_isomorphism:B:prep:1} above, we have $\varphi \in S^c_s(\xi_2,\xi_{22})\cap S_s(\xi_2,\xi_{22})=\{0\}$.\\
\end{proof}

\begin{proof}[Proof of Proposition \ref{prop:decomposition_cB0}]
It suffices to prove
\begin{equation*}
\iota_{1\otimes2}(\cB^0)=S^0(\xi_1,\xi_{11},\xi^{(1)}_{12}) \otimes S^0(\xi_2,\xi_{22}).
\end{equation*}
First, we show $\iota_{1\otimes2}(\cB^0) \subset S^0(\xi_1,\xi_{11},\xi^{(1)}_{12}) \otimes S^0(\xi_2,\xi_{22})$. Put $\varphi \in \cB^0$. It is clear that $\iota_{1\otimes2}(\varphi) \in S(\xi_1,\xi_{11},\xi^{(1)}_{12}) \otimes S^0(\xi_2,\xi_{22})$. We assume that
\begin{equation*}
\iota_{1\otimes2}(\varphi) \in S(\xi_1,\xi_{11},\xi^{(1)}_{12}) \otimes S^0(\xi_2,\xi_{22})-S^0(\xi_1,\xi_{11},\xi^{(1)}_{12}) \otimes S^0(\xi_2,\xi_{22}),
\end{equation*}
namely $\varphi$ contains a term of the form $\varphi' \xi_1  \varphi''$, where $\varphi' \in S(\xi_1,\xi_{11},\xi_{12}), \varphi'' \in S(\xi_2,\xi_{22})$. Thus by Lemma \ref{lem:iota_isomorphism:B:prep} \eqref{lem:iota_isomorphism:B:prep:2}, $\varphi$ contains $\varphi'  \varphi''  \xi_1$. This means that $\varphi$ does not belong to $\cB^0$. Therefore we obtain $\iota_{1\otimes2}^{(1)}(\varphi) \in S^0(\xi_1,\xi_{11},\xi^{(1)}_{12}) \otimes S^0(\xi_2,\xi_{22})$.

Next, we prove that $\iota_{1\otimes2}(\cB^0) \supset S^0(\xi_1,\xi_{11},\xi^{(1)}_{12}) \otimes S^0(\xi_2,\xi_{22})$. Since $\iota_{1\otimes2}$ is an isomorphism, for any $\psi \in S^0(\xi_1,\xi_{11},\xi^{(1)}_{12}) \otimes S^0(\xi_2,\xi_{22})$, there exists a unique element $\varphi \in \cB$ such that $\iota_{1\otimes2}(\varphi)=\psi$. We assume that $\varphi \in \cB-\cB^0$, that is, $\varphi$ has a term ended with $\xi_1$ or $\xi_2$. By using Lemma \ref{lem:iota_isomorphism:B:prep} (\ref{lem:iota_isomorphism:B:prep:3}), we see that $\iota_{1\otimes2}(\varphi)$ is not an element of $S^0(\xi_1,\xi_{11},\xi^{(1)}_{12}) \otimes S^0(\xi_2,\xi_{22})$.

We have thus completed the proof of Proposition \ref{prop:decomposition_cB0}.
\end{proof}

For one variable, it is well known that $S(\xi_1,\xi_{11})$ is the shuffle polynomial ring over $S^0(\xi_1,\xi_{11})=\C\bnull+S(\xi_1,\xi_{11})\xi_{11}$, namely $S(\xi_1,\xi_{11})= S^0(\xi_1,\xi_{11})[\xi_1]$ (Reutenauer \cite{R}). Moreover $S(\xi_1,\xi_{11},\xi^{(1)}_{12})= S^0(\xi_1,\xi_{11},\xi^{(1)}_{12})[\xi_1]$ also holds. As a corollary of Proposition \ref{prop:decomposition_cB} and \ref{prop:decomposition_cB0}, one can show the following ``two variables'' analogue.

\begin{cor}
We have
\begin{equation}
\cB=\cB^0[\xi_1,\xi_2].
\end{equation}
\end{cor}

\vspace{1\baselineskip}

We note that the maps $\iota_{i\otimes j}$ can be obtained by the following way:
\begin{enumerate}
\item Picking up the terms which have the form $\psi_i \psi_j \; (\psi_i \in S(\xi_i,\xi_{ii},\xi_{12}),\; \psi_j \in S(\xi_j,\xi_{jj}))$.
\item Changing each $\psi_i \psi_j$ to $\psi_i \otimes \psi_j \in S(\xi_i,\xi_{ii},\xi_{12})\otimes S(\xi_j,\xi_{jj})$.
\item Replacing $\xi_{12}$ to $\xi_{12}^{(i)}$.\\
\end{enumerate}

\begin{figure}[h]
\setlength{\unitlength}{1.5cm}
\begin{picture}(0,4.3)(-2,0)
\put(0,0){\scalebox{0.75}{\includegraphics{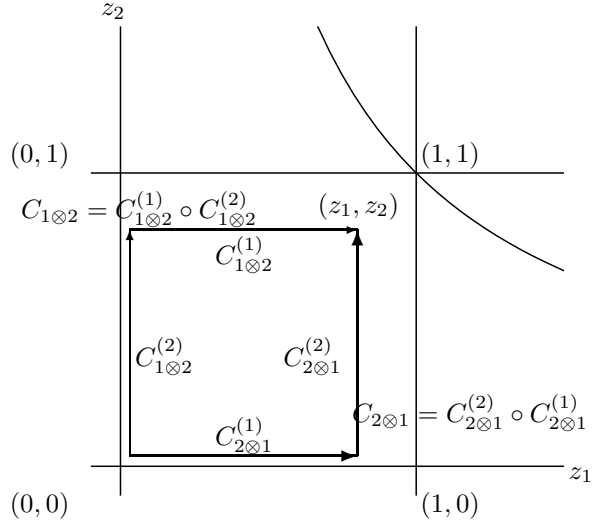}}}
\put(-0.5,0){$(0,0)$}
\put(3.1,0){$(1,0)$}
\put(-0.5,3.1){$(0,1)$}
\put(3.1,3.1){$(1,1)$}
\put(4.4,0.3){$z_1$}
\put(0.3,4.4){$z_2$}
\put(2.2,2.65){$(z_1,z_2)$}
\put(1.3,0.6){$C_{2\otimes1}^{(1)}$}
\put(1.9,1.3){$C_{2\otimes1}^{(2)}$}
\put(0.6,1.3){$C_{1\otimes2}^{(2)}$}
\put(1.3,2.2){$C_{1\otimes2}^{(1)}$}
\put(2.5,0.8){$C_{2\otimes 1}=C_{2\otimes1}^{(2)}\circ C_{2\otimes1}^{(1)}$}
\put(-0.4,2.6){$C_{1\otimes 2}=C_{1\otimes2}^{(1)}\circ C_{1\otimes2}^{(2)}$}

\thicklines
\put(0.55,0.5){\vector(1,0){2}}
\put(2.55,0.5){\vector(0,1){2}}
\thinlines
\put(0.55,0.5){\vector(0,1){2}}
\put(0.55,2.5){\vector(1,0){2}}
\end{picture}
\setlength{\unitlength}{1cm}
\caption{the contours $C_{1\otimes2}, C_{2\otimes1}$}
\label{fig:C12_C21}
\end{figure}

For $0<|z_1|<1$ and $0<|z_2|<1$, we denote by $C_{1\otimes2}^{(2)}$ the contour from $(0,0)$ to $(0,z_2)$ on $\{0\}\times \{z\in \C\;|\; |z|<1\}$, and $C_{1\otimes2}^{(1)}$ the contour from $(0,z_2)$ to $(z_1,z_2)$ on $\{z\in \C\;|\; |z|<1\}\times\{z_2\}$. We also denote by $C_{2\otimes1}^{(1)}$ the contour from $(0,0)$ to $(z_1,0)$ on $\{z\in \C\;|\; |z|<1\}\times\{0\}$, and $C_{2\otimes1}^{(2)}$ the contour from $(z_1,0)$ to $(z_1,z_2)$ on $\{z_1\}\times\{z\in \C\;|\; |z|<1\}$.

We define the integral contours $C_{1\otimes 2}$ and $C_{2\otimes 1}$ by $C_{1\otimes 2}=C_{1\otimes2}^{(1)}\circ C_{1\otimes2}^{(2)}$, $C_{2\otimes 1}=C_{2\otimes1}^{(2)}\circ C_{2\otimes1}^{(1)}$, where the composition of paths $C \circ C'$ is defined by connecting $C$ after $C'$ (see Figure \ref{fig:C12_C21}).

The projection $\iota_{i\otimes j}$ corresponds to the contour $C_{i\otimes j}$ as the following proposition.

\begin{prop}\label{prop:contour_C12_C21}
If $\varphi \in \cB^0$, iterated integrals $\int_{C_{1\otimes 2}}\varphi$ and $\int_{C_{1\otimes 2}}\varphi$ make sense and are given as follows:
\begin{align}
\int_{C_{1\otimes2}}\!\!\varphi\;\;&=\int_{1\otimes2}\!\!\iota_{1\otimes2}(\varphi),\\
\int_{C_{2\otimes1}}\!\!\varphi\;\;&=\int_{2\otimes1}\!\!\iota_{2\otimes1}(\varphi),
\end{align}
where the integrals of the right hand sides stand for
\begin{equation}
\int_{1\otimes2}\psi_1 \otimes \psi_2 := \int_{C_{1\otimes2}}\psi_1\psi_2 = \int_{z_1=0}^{z_1}\!\!\psi_1 \int_{z_2=0}^{z_2}\!\!\!\psi_2 \label{def_integral_12}
\end{equation}
for $\psi_1 \otimes \psi_2 \in S^0(\xi_1,\xi_{11},\xi_{12}^{(1)})\otimes S^0(\xi_2,\xi_{22})$, and
\begin{equation}
\int_{2\otimes1}\psi_2 \otimes \psi_1 := \int_{C_{2\otimes1}}\psi_2\psi_1 = \int_{z_2=0}^{z_2}\!\!\psi_2 \int_{z_1=0}^{z_1}\!\!\!\psi_1 \label{def_integral_21}
\end{equation}
for $\psi_2 \otimes \psi_1 \in S^0(\xi_2,\xi_{22},\xi_{12}^{(2)})\otimes S^0(\xi_1,\xi_{11})$.

\end{prop}

\begin{proof}
The definition of $\iota$ says that $\iota_{i\otimes j}$ is the map picking up terms which do not vanish for integration on $C_{i\otimes j}$.
\end{proof}

\section{The hyperlogarithms of the type $\cM_{0,5}$}\label{sec:hyperlogarithm}
For the integral \eqref{def_integral_12} and \eqref{def_integral_21}, the integrals
\begin{align*}
&\int_{0}^{z_1}w \qquad (w \in S^0(\xi_1,\xi_{11}))
\intertext{and}
&\int_{0}^{z_2}w \qquad (w \in S^0(\xi_2,\xi_{22}))
\end{align*}
are nothing but multiple polylogarithms of one variable appeared in Section \ref{sec:KZ1}
\begin{equation*}
\int_{0}^{z_i}w=\Li(w;z_i) \qquad (w \in S^0(\xi_i,\xi_{ii})).
\end{equation*}
In this section, we consider the integrals
\begin{align*}
\int_{0}^{z_1}w \qquad (w \in S^0(\xi_1,\xi_{11},\xi_{12}^{(1)}))
\intertext{and}
\int_{0}^{z_2}w \qquad (w \in S^0(\xi_2,\xi_{22},\xi_{12}^{(2)})).
\end{align*}
They are referred to as hyperlogarithms of the type $\cM_{0,5}$.

For $w=\xi_1^{k_1-1}\omega_1\cdots \xi_1^{k_r-1}\omega_r \in S^0(\xi_1,\xi_{11},\xi_{12}^{(1)})$ \quad ($\omega_i\in\{\xi_{11},\xi^{(1)}_{12}\}$), we denote by $L(w;z_1)=\int_0^{z_1}w$ the iterated integral of $w$ from 0 to $z_1$. This is the hyperlogarithm
\begin{equation*}
L(w;z_1)=L({}^{k_1}\alpha_1\cdots{}^{k_r}\alpha_r;z_1),
\end{equation*}
where $\alpha_i=1$ or $z_2$ with respect to $\omega_i=\xi_{11}$ or $\xi_{12}^{(1)}$. We call this ``hyperlogarithm of the main variable $z_1$ with the singular points $\{0,1,\frac{1}{z_2},\infty\}$''. This function is a many-valued analytic function on $\bP^1-\{0,1,\frac{1}{z_2},\infty\}$ and has a Taylor expansion 
\begin{equation}
L(w;z_1)=\sum_{m_1>\cdots>m_r>0}\!\!\!\!\!\!\frac{\alpha_1^{m_1-m_2}\alpha_2^{m_2-m_3}\cdots \alpha_r^{m_r}}{m_1^{k_1} \cdots m_r^{k_r}} z_1^{m_1}
\end{equation}
on a neighborhood of $z_1=0$.

Particularly, if $w \in S^0(\xi_1,\xi_{11})$, the hyperlogarithm $L(w;z_1)$ is a multiple polylogarithm of one variable
\begin{equation}
L({}^{k_1}1\cdots{}^{k_r}1;z_1)=\Li_{k_1,\ldots,k_r}(z_1),
\end{equation}
and if $w_1,\ldots,w_i=\xi_{11}, w_{i+1},\ldots,w_r=\xi_{12}^{(1)}$, it is a multiple polylogarithm of two variables appeared in Section \ref{sec:HP}
\begin{equation}
L({}^{k_1}1\cdots{}^{k_i}1{}^{k_{i+1}}z_2\cdots{}^{k_r}z_2;z_1)=\Li_{k_1,\ldots,k_r}(i,r-i;z_1,z_2).
\end{equation}

In the same way, for $w \in S^0(\xi_2,\xi_{22},\xi_{12}^{(2)})$, we call $L(w;z_2)=\int_0^{z_2}w$ ``hyperlogarithms of the main variable $z_2$ with the singular points $\{0,1,\frac{1}{z_1},\infty\}$''.\\

Under this convention, the iterated integral of an element of $\cB^0$ on $C_{1\otimes2}$ (resp. $C_{2\otimes1}$) can be written as the product of hyperlogarithm of $z_1$ (resp. $z_2$) and MPL1 of $z_2$ (resp. $z_1$).\\

Multiple polylogarithms satisfy the following recursive relations which can be proved easily by using the series expansion.
\begin{lem} \label{lem:2-mpl}
$\Li_{k_1,\ldots,k_{i+j}}(i,j;z_1,z_2)$ satisfies the following differential recursive relations.
\begin{align}
&\frac{\partial}{\partial z_1}\Li_{k_1,\ldots,k_{i+j}}(i,j;z_1,z_2) \label{2MPL_dz}\\
&\quad =\begin{cases}
\ds \frac{z_2}{1-z_1z_2}\Li_{k_2,\ldots,k_{j}}(0,j-1;z_1,z_2) \qquad & (i=0,k_1=1),\\
\ds \frac{1}{1-z_1}\Li_{k_2,\ldots,k_{i+j}}(i-1,j;z_1,z_2) & (i>0,k_1=1),\\
\ds \frac{1}{z_1}\Li_{k_1-1,k_2,\ldots,k_{i+j}}(i,j;z_1,z_2) & (k_1>1),
\end{cases} \notag \\
&\frac{\partial}{\partial z_2}\Li_{k_1,\ldots,k_{i+j}}(i,j;z_1,z_2) \label{2MPL_dw} \\
&\; =\begin{cases}
\ds \frac{z_1}{1-z_1z_2}\Li_{k_2,\ldots,k_{j}}(0,j-1;z_1,z_2) & (i=0,k_1=1),\\
\ds \frac{1}{1-z_2} \Li_{k_1,\ldots,k_i,k_{i+2},\ldots,k_{i+j}}(i,j-1;z_1,z_2)\\
\ds \quad -\frac{1}{1-z_2} \Li_{k_1,\ldots,k_i,k_{i+2},\ldots,k_{i+j}}(i-1,j;z_1,z_2)\\
\ds \quad \qquad\;\; -\frac{1}{z_2} \Li_{k_1,\ldots,k_i,k_{i+2},\ldots,k_{i+j}}(i-1,j;z_1,z_2)
& (i>0,k_{i+1}=1),\\
\ds \frac{1}{z_2}\Li_{k_1,\ldots,k_i,k_{i+1}-1,k_{i+2},\ldots,k_{i+j}}(i,j;z_1,z_2) & (k_{i+1}>1).
\end{cases} \notag
\end{align}
\end{lem}

\section{The generalized harmonic product relations and the harmonic product of MPLs}\label{sec:GHPR}

Now we obtain immediately the following theorem.

\begin{thm}[The generalized harmonic product relations]\label{thm:generalized_harmonic_product_relation}
For each $\varphi \in \cB^0$, the relation among hyperlogarithms of the type $\cM_{0,5}$
\begin{equation}\label{GHPR}
\int_{1\otimes2}\iota_{1\otimes2}(\varphi) = \int_{2\otimes1}\iota_{2\otimes1}(\varphi)
\end{equation}
holds.
\end{thm}

\begin{proof}
Since an iterated integral of $\varphi \in \cB^0$ depends only on the homotopy class of the integral path, we have
\begin{equation} \label{proof_of_GHPR}
\int_{C_{1\otimes2}}\varphi=\int_{C_{2\otimes1}}\varphi.
\end{equation}
Applying Proposition \ref{prop:contour_C12_C21} to \eqref{proof_of_GHPR}, we have the theorem.
\end{proof}

Particularly, applying the theorem to $w \in S^0(\xi_1,\xi_{11},\xi_{12}^{(1)})$, we have
\begin{equation}\label{GHPR'}
L(w;z_1)=\int_{2\otimes1}\iota_{2\otimes1}\circ\iota_{1\otimes2}^{-1}(w\otimes\bnull).
\end{equation}
This equation says that the hyperlogarithm $L(w;z_1)$ of the main variable $z_1$ can be represented in terms of hyperlogarithms of the main variable $z_2$ and multiple polylogarithms of $z_1$.

Relations \eqref{GHPR} among hyperlogarithms of the type $\cM_{0,5}$ for all $\varphi \in \cB^0$ are referred to as the generalized harmonic product relations.

We show some examples of the generalized harmonic product relations:
\begin{itemize}
\item For $w=\xi_{11}\xi_{12}^{(1)}$, we have
\begin{equation*}
\iota_{1\otimes2}^{-1}(w\otimes\bnull)=\xi_{11}\xi_{12}+\xi_{22}\xi_{11}-\xi_{22}\xi_{12}-\xi_2\xi_{12}.
\end{equation*}
The generalized harmonic product relation for $\iota_{1\otimes2}^{-1}(\xi_{11}\xi_{12}^{(1)}\otimes\bnull)$ is
\begin{equation*}
\Li_{1,1}(1,1;z_1,z_2)=\Li_1(z_2)\Li_1(z_1)-\Li_{1,1}(1,1;z_2,z_1)-\Li_{2}(0,1:z_2,z_1).
\end{equation*}
This is the simplest case of the harmonic product of MPLs.

\item For $w=\xi_{12}^{(1)}\xi_{11}$, we have
\begin{equation*}
\iota_{1\otimes2}^{-1}(w\otimes\bnull)=\xi_{12}\xi_{11}-\xi_{22}\xi_{11}+\xi_{22}\xi_{12}+\xi_2\xi_{12}.
\end{equation*}
The generalized harmonic product relation for $\iota_{1\otimes2}^{-1}(\xi_{12}^{(1)}\xi_{11}\otimes\bnull)$ is
\begin{align*}
L({}^{1}z_2{}^{1}1;z_1)&=\Li_1(0,1;z_2,z_1)\Li_1(z_1)-\Li_1(z_2)\Li_1(z_1)\\
&\qquad \qquad +\Li_{1,1}(1,1;z_2,z_1)+\Li_{2}(0,1;z_2,z_1).
\end{align*}

\item For $w=\xi_{12}^{(1)}\xi_{11}\xi_{12}^{(1)}$, we have
\begin{align*}
\iota_{1\otimes2}^{-1}(w\otimes\bnull)\!&=\!\xi_{12}\xi_{11}\xi_{12}+\xi_{12}\xi_{22}\xi_{11}+\xi_{22}\xi_{12}\xi_{11}-\xi_{22}\xi_{11}\xi_{12}\notag\\
&\qquad -2\xi_{22}\xi_{22}\xi_{11}+2\xi_{22}\xi_{22}\xi_{12}+2\xi_{22}\xi_2\xi_{12}-\xi_{12}\xi_{22}\xi_{12}-\xi_{12}\xi_2\xi_{12}.
\end{align*}
The generalized harmonic product relation for $\iota_{1\otimes2}^{-1}(\xi_{12}^{(1)}\xi_{11}\xi_{12}^{(1)}\otimes\bnull)$ is
\begin{align*}
&L({}^1z_2{}^11{}^1z_2;z_1)\\
&\qquad =-2\Li_{1,1}(z_2)\Li_1(z_1)+2\Li_{1,1,1}(2,1;z_2,z_1)\\
&\qquad\phantom{=} +2\Li_{1,2}(1,1;z_2,z_1)+L({}^1z_1{}^11;z_2)\Li_1(z_1)\\
&\qquad\phantom{=} +\Li_{1,1}(1,1;z_2,z_1)\Li_1(z_1)-L({}^1z_1{}^11{}^1z_1;z_2)-\Li_{1,2}(0,2;z_2,z_1).
\end{align*}
\end{itemize}

Next, we consider the case of $\varphi=\iota_{1\otimes2}^{-1}(w \otimes \bnull) \in \cB^0$ for
\begin{equation*}
w=\xi_1^{k_1-1}\xi_{11}\cdots\xi_1^{k_i-1}\xi_{11}\xi_1^{k_{i+1}-1}\xi_{12}^{(1)}\cdots \xi_1^{k_r-1}\xi_{12}^{(1)} \in S^0(\xi_1,\xi_{11},\xi_{12}^{(1)}).
\end{equation*}
Put $\varphi_{k_1,\ldots,k_r}(i,r-i)=\iota_{1\otimes2}^{-1}(w \otimes \bnull)$ and $\varphi_{\emptyset}(0,0)=\bnull$.

\begin{prop}\label{prop:phi_Li_z1}
The following identity in $\cB$ holds:
\begin{multline}\label{phi_Li_z1}
\varphi_{k_1,\ldots,k_r}(i,r-i)\\
=\begin{cases}
\xi_{11}\varphi_{k_2,\ldots,k_r}(r-1,0)&(i=r,k_1=1),\\
\xi_1\varphi_{(k_1-1),k_2,\ldots,k_r}(r,0)&(i=r,k_1>1),\\
\xi_{12}\varphi_{k_2,\ldots,k_r}(0,r-1)&(i=0,k_1\!=\!k_{i+1}\!= 1),\\
\xi_{11}\varphi_{k_2,\ldots,k_r}(i-1,r-i)\\
\; +\xi_{22}\varphi_{k_1,\ldots,k_i,k_{i+2},\ldots,k_r}(i,r-i-1)\\
\; -(\xi_{22}+\xi_2)\varphi_{k_1,\ldots,k_i,k_{i+2},\ldots,k_r}(i-1,r-i)&(r>i>0, k_1\!=\!k_{i+1}\!=1),\\
\xi_{11}\varphi_{k_2,\ldots,k_r}(i-1,r-i)\\
\; +\xi_2\varphi_{k_1,\ldots,k_i,(k_{i+1}-1),k_{i+2},\ldots,k_r}(i,r-i)&(r>i>0, k_1\!=1, k_{i+1}\!>\!1),\\
\xi_1\varphi_{(k_1-1),k_2,\ldots,k_r}(i,r-i)\\
\; +\xi_{22}\varphi_{k_1,\ldots,k_i,k_{i+2},\ldots,k_r}(i,r-i-1)\\
\; -(\xi_{22}+\xi_2)\varphi_{k_1,\ldots,k_i,k_{i+2},\ldots,k_r}(i-1,r-i)&(r>i>0, k_1>1, k_{i+1}\!=1),\\
\xi_1\varphi_{(k_1-1),k_2,\ldots,k_r}(i,r-i)\\
\; +\xi_2\varphi_{k_1,\ldots,k_i,(k_{i+1}-1),k_{i+2},\ldots,k_r}(i,r-i)&(r>i\ge 0, k_1\!>\!1, k_{i+1}\!>\!1).
\end{cases}
\end{multline}
\end{prop}

\begin{proof}
We show this by induction on the weight $k=k_1+\cdots+k_r$. If $k=1$, we have
\begin{align*}
\varphi_1(1,0)&=\iota_{1\otimes2}^{-1}(\xi_{11}\otimes\bnull)=\xi_{11}=\xi_{11}\bnull=\xi_{11}\varphi_{\emptyset}(0,0),\\
\varphi_1(0,1)&=\iota_{1\otimes2}^{-1}(\xi^{(1)}_{12}\otimes\bnull)=\xi_{12}=\xi_{12}\bnull=\xi_{12}\varphi_{\emptyset}(0,0).
\end{align*}
Therefore the equation \eqref{phi_Li_z1} holds.

We assume that the equation \eqref{phi_Li_z1} holds for all $l_1+\cdots+l_s<k$. By the definition of $\iota_{1\otimes2}$, we obtain $\iota_{1\otimes2}(\text{LHS})=\iota_{1\otimes2}(\text{RHS})=w\otimes\bunit$ for all cases. Thus it suffices to prove the RHS of \eqref{phi_Li_z1} belongs to $\cB_2 \cB^0_{k-2} \subset \cB^0$ for $k_1+\cdots+k_r=k$. For $i=0, k_1=k_{i+1}=1$, we have
\begin{align*}
\text{RHS}&=\xi_{12}\varphi_{k_2,\ldots,k_r}(0,r-1)\\
&=\begin{cases}
\xi_{12}\xi_{12}\varphi_{k_3,\ldots,k_r}(0,r-2)&(k_2=1),\\
\xi_{12}(\xi_1+\xi_2)\varphi_{(k_2-1),k_3,\ldots,k_r}(0,r-1)&(k_2>1)
\end{cases}\\
&\in \cB_2\cB_{k-2}.
\end{align*}
For other cases, we can show that the RHS of \eqref{phi_Li_z1} belongs to $\cB_2 \cB^0_{k-2}$ in the same fashion.
\end{proof}

\begin{cor}\label{cor:GHPR_i0type}
We have
\begin{equation}
\varphi_{k_1,\ldots,k_r}(0,r)=(\xi_1+\xi_2)^{k_1-1}\xi_{12}\cdots(\xi_1+\xi_2)^{k_r-1}\xi_{12}.
\end{equation}
Moreover the generalized harmonic product relation for $\varphi_{k_1,\ldots,k_r}(0,r)$ is the trivial relation
\begin{equation}
\Li_{k_1,\ldots,k_r}(0,r;z_1,z_2)=\Li_{k_1,\ldots,k_r}(0,r;z_2,z_1).
\end{equation}
\end{cor}

\begin{proof}
Applying Proposition \ref{prop:phi_Li_z1} to $\varphi_{k_1,\ldots,k_r}(0,r)$ recursively, we have the first formula. The second claim follows immediately from the first formula.
\end{proof}

Let $\cMPL=\{\varphi_{k_1,\ldots,k_r}(i,r-i)\} \subset \cB^0$ be a subset of elements such that
\begin{equation*}
\iota_{1\otimes2}^{-1}(\xi_1^{k_1-1}\xi_{11}\cdots\xi_1^{k_i-1}\xi_{11}\xi_1^{k_{i+1}-1}\xi_{12}^{(1)}\cdots \xi_1^{k_r-1}\xi_{12}^{(1)} \otimes \bnull).
\end{equation*}
Set $\cMPL_k=\cMPL\cap\cB_k$. The expression \eqref{phi_Li_z1} defines five maps
\begin{equation*}
\fd_{\mu}: \cMPL_k \to \cMPL_{k-1}
\end{equation*}
for $\mu \in \{1,11,2,22,12\}$ by
\begin{equation*}
\varphi_{k_1,\ldots,k_r}(i,r-i)=\sum_{\mu \in \{1,11,2,22,12\}}\xi_{\mu}\fd_{\mu}(\varphi_{k_1,\ldots,k_r}(i,r-i)).
\end{equation*}

\begin{cor}
\begin{align}
d\Li_{k_1,\ldots,k_r}(i,r-i;z_1,z_2)&=\!\!\!\!\!\sum_{\mu\in \{1,11,2,22,12\}}\xi_\mu\int_{1\otimes2}\iota_{1\otimes2}(\fd_\mu(\varphi_{k_1,\ldots,k_r}(i,r-i))).
\end{align}
\end{cor}

\begin{proof}
It follows immediately from Lemma \ref{lem:2-mpl} and the computation
\begin{multline*}
d\Li_{k_1,\ldots,k_r}(i,r-i;z_1,z_2)\\
=\frac{\partial \Li_{k_1,\ldots,k_r}(i,r-i;z_1,z_2)}{\partial z_1}dz_1+\frac{\partial\Li_{k_1,\ldots,k_r}(i,r-i;z_1,z_2)}{\partial z_2}dz_2.
\end{multline*}
\end{proof}

Now we can express recursively the generalized harmonic product relation \eqref{GHPR} for $\varphi_{k_1,\ldots,k_r}(i,r-i)$ as follows:
\begin{align*}
\int_{2\otimes1}\iota_{2\otimes1}(\varphi_{k_1,\ldots,k_r}(i,r-i))&=\int_{C_{2\otimes1}}\varphi_{k_1,\ldots,k_r}(i,r-i)\\
&=\int_{C_{2\otimes1}}\sum_{\mu}\xi_{\mu}\int_{C_{2\otimes1}}\fd_{\mu}(\varphi_{k_1,\ldots,k_r}(i,r-i))\\
&=\int_{C_{2\otimes1}}\sum_{\mu}\xi_{\mu}\int_{2\otimes1}\iota_{2\otimes1}(\fd_{\mu}(\varphi_{k_1,\ldots,k_r}(i,r-i)))\\
&=\int_{C_{2\otimes1}}\sum_{\mu}\xi_{\mu}\int_{1\otimes2}\iota_{1\otimes2}(\fd_{\mu}(\varphi_{k_1,\ldots,k_r}(i,r-i))).
\end{align*}
The last equation follows from the generalized harmonic product relations for the weight $k_1+\cdots+k_r-1$. Hence we have
\begin{align*}
\int_{2\otimes1}\iota_{2\otimes1}(\varphi_{k_1,\ldots,k_r}(i,r-i))&=\int_{C_{2\otimes1}}\sum_{\mu}\xi_{\mu}\int_{1\otimes2}\iota_{1\otimes2}(\fd_{\mu}(\varphi_{k_1,\ldots,k_r}(i,r-i))) \notag\\
&=\int_{C_{2\otimes1}}d\Li_{k_1,\ldots,k_r}(i,r-i;z_1,z_2).
\end{align*}
On the other hand, we obtain
\begin{equation*}
\int_{1\otimes2}\iota_{1\otimes2}(\varphi_{k_1,\ldots,k_r}(i,r-i))=\Li_{k_1,\ldots,k_r}(i,r-i;z_1,z_2).
\end{equation*}
Therefore the generalized harmonic product relation
\begin{equation*}
\int_{1\otimes2}\iota_{1\otimes2}(\varphi_{k_1,\ldots,k_r}(i,r-i))=\int_{2\otimes1}\iota_{2\otimes1}(\varphi_{k_1,\ldots,k_r}(i,r-i))
\end{equation*}
is equivalent to the recursive relation
\begin{equation}\label{GHPR_recursive}
\Li_{k_1,\ldots,k_r}(i,r-i;z_1,z_2)=\int_{C_{2\otimes1}}d\Li_{k_1,\ldots,k_r}(i,r-i;z_1,z_2).
\end{equation}

Using the fact above, one can calculate the generalized harmonic product relation for $\Li_{k_1,\ldots,k_r}(i,r-i;z_1,z_2)$ recursively without determining a concrete expression of $\varphi_{k_1,\ldots,k_r}(i,r-i)$ as follows:
\begin{enumerate}
\item Write
\begin{equation*}
d\Li_{k_1,\ldots,k_r}(i,r-i;z_1,z_2)=\sum_{\mu \in \{1,11,2,22,12\}}\xi_{\mu}f_{\mu},
\end{equation*}
where $f_{\mu}$ for $\mu \in \{1,11,2,22,12\}$ is a linearly combination of MPL2 with the main variable $z_1$ with the weight $k_1+\cdots+k_r-1$.
\item By using the generalized harmonic product relations for weight $k_1+\cdots+k_r-1$, convert $f_1,\ldots,f_{12}$ to the product of MPL2 with the main variable $z_2$ and MPL1 of $z_1$. We denote the results by
\begin{equation*}
d\Li_{k_1,\ldots,k_r}(i,r-i;z_1,z_2)=\sum_{\mu \in \{1,11,2,22,12\}}\xi_{\mu}g_{\mu}
\end{equation*}
where $g_{\mu}$ for $\mu \in \{1,11,2,22,12\}$ is a linearly combination of the product of MPL2 with the main variable $z_2$ and MPL1 of $z_1$.
\item Compute $\ds \int_{C_{2\otimes1}}\xi_1g_1+\cdots\xi_{12}g_{12}$ as follows. For $\bk=(k_1,\ldots,k_r)$ and $\bl=(l_1,\ldots,l_s)$, we have
\begin{align*}
&\int_{C_{2\otimes1}}\xi_1\Li_{\bk}(i,r-i;z_2,z_1)\Li_{\bl}(z_1)=\begin{cases}0&(\bk\neq \emptyset),\\\Li_{(l_1+1),l_2,\ldots,l_s}(z_1)&(\bk= \emptyset),\end{cases}\\
&\int_{C_{2\otimes1}}\xi_{11}\Li_{\bk}(i,r-i;z_2,z_1)\Li_{\bl}(z_1)=\begin{cases}0&(\bk\neq \emptyset),\\\Li_{1,l_1,\ldots,l_s}(z_1)&(\bk= \emptyset),\end{cases}\\
&\int_{C_{2\otimes1}}\xi_2\Li_{\bk}(i,r-i;z_2,z_1)\Li_{\bl}(z_1)=\Li_{(k_1+1),k_2,\ldots,k_r}(i,r-i;z_2,z_1)\Li_{\bl}(z_1),\\
&\int_{C_{2\otimes1}}\xi_{22}\Li_{\bk}(i,r-i;z_2,z_1)\Li_{\bl}(z_1)=\Li_{1,k_1,\ldots,k_r}(i+1,r-i;z_2,z_1)\Li_{\bl}(z_1),\\
&\int_{C_{2\otimes1}}\xi_{12}g_{12}=\begin{cases}
\Li_{1,k_2,\ldots,k_r}(0,r;z_2,z_1)&(i=0,k_1=1)\\
0&(\text{otherwise}).
\end{cases}
\end{align*}
The first four formulas are derived from the definition of $\iota_{2\otimes1}$ immediately. The fifth equation holds by virtue of Corollary \ref{cor:GHPR_i0type}.
\end{enumerate}

For instance, we calculate the generalized harmonic product for the MPL2 $\Li_{2,1}(1,1;z_1,z_2)$. Since
\begin{multline*}
d\Li_{2,1}(1,1;z_1,z_2)\\
=\xi_1\Li_{1,1}(1,1;z_1,z_2)+\xi_{22}\Li_{2}(1,0;z_1,z_2)-(\xi_{22}+\xi_2)\Li_{2}(0,1;z_1,z_2),
\end{multline*}
we prepare the generalized harmonic product relations for $\Li_{1,1}(1,1;z_1,z_2)$, $\Li_{2}(1,0;z_1,z_2)$ and $\Li_{2}(0,1;z_1,z_2)$. For $\Li_{2}(1,0;z_1,z_2)$ and $\Li_{2}(0,1;z_1,z_2)$, we obtain easily
\begin{align*}
\Li_{2}(1,0;z_1,z_2)&=\Li_{2}(z_1),\\
\Li_{2}(0,1;z_1,z_2)&=\Li_{2}(0,1;z_2,z_1).
\end{align*}
For $\Li_{1,1}(1,1;z_1,z_2)$, we have
\begin{align*}
&\Li_{1,1}(1,1;z_1,z_2)\\
&=\int_{C_{2\otimes1}}d\Li_{1,1}(1,1;z_1,z_2)\\
&=\int_{C_{2\otimes1}}\Big(\xi_{11}\Li_{1}(0,1;z_1,z_2)+\xi_{22}\Li_{1}(1,0;z_1,z_2)\\*[-2ex]
&\phantom{\int_{C_{2\otimes1}}\Big(\xi_{11}\Li_{1}}-(\xi_{22}+\xi_2)\Li_{1}(0,1;z_1,z_2)\Big)\\
&=\int_{C_{2\otimes1}}\Big(\xi_{11}\Li_{1}(0,1;z_2,z_1)+\xi_{22}\Li_{1}(z_1)-(\xi_{22}+\xi_2)\Li_{1}(0,1;z_2,z_1)\Big)\\
&=0+\Li_1(1,0;z_2,z_1)\Li_{1}(z_1)-\Li_{1,1}(1,1;z_2,z_1)-\Li_{2}(0,1;z_2,z_1).
\end{align*}
Thus we obtain
\begin{align*}
&\Li_{2,1}(1,1;z_1,z_2)\\
&=\int_{C_{2\otimes1}}d\Li_{1,1}(2,1;z_1,z_2)\\
&=\int_{C_{2\otimes1}}\Big(\xi_1\Li_{1,1}(1,1;z_1,z_2)+\xi_{22}\Li_{2}(1,0;z_1,z_2)-(\xi_{22}+\xi_2)\Li_{2}(0,1;z_1,z_2)\Big)\\
&=\int_{C_{2\otimes1}}\Big(\xi_1(\Li_1(1,0;z_2,z_1)\Li_{1}(z_1)-\Li_{1,1}(1,1;z_2,z_1)-\Li_{2}(0,1;z_2,z_1))\\*[-2ex]
&\phantom{=\int_{C_{2\otimes1}}\Big(\xi_1\Li_{1,1}}+\xi_{22}\Li_{2}(z_1)-(\xi_{22}+\xi_2)\Li_{2}(0,1;z_2,z_1)\Big)\\
&=0+\Li_1(1,0;z_2,z_1)\Li_{2}(z_1)-\Li_{1,2}(1,1;z_2,z_1)-\Li_{3}(0,1;z_2,z_1).
\end{align*}
Therefore the generalized harmonic product relation for $\Li_{2,1}(1,1;z_1,z_2)$ is 
\begin{equation*}
\Li_{2,1}(1,1;z_1,z_2)=\Li_1(1,0;z_2,z_1)\Li_{2}(z_1)-\Li_{1,2}(1,1;z_2,z_1)-\Li_{3}(0,1;z_2,z_1).
\end{equation*}
This is nothing but the harmonic product of multiple polylogarithms
\begin{equation*}
\Li_{2}(z_1)\Li_1(z_2)=\Li_{2,1}(1,1;z_1,z_2)+\Li_{3}(z_1z_2)+\Li_{1,2}(1,1;z_2,z_1).
\end{equation*}

\vspace{1\baselineskip}

In general, we can show that the generalized harmonic product relations contains the harmonic product of multiple polylogarithms. Namely

\begin{thm}\label{thm:ghpl_harmonic_product}
For
\begin{equation*}
w=\xi_1^{k_1-1}\xi_{11}\cdots\xi_1^{k_i-1}\xi_{11}\xi_1^{k_{i+1}-1}\xi_{12}^{(1)}\cdots \xi_1^{k_r-1}\xi_{12}^{(1)} \in S^0(\xi_1,\xi_{11},\xi_{12}^{(1)}),
\end{equation*}
the generalized harmonic product relations
\begin{equation}\label{GHPR''}
\Li_{k_1,\ldots,k_r}(i,r-i;z_1,z_2)=\int_{2\otimes1}\iota_{2\otimes1}\circ\iota_{1\otimes2}^{-1}(w\otimes\bnull)
\end{equation}
yield relations among multiple polylogarithms of two variables. Moreover $\Z$-linear combinations of these relations contain the harmonic product of multiple polylogarithms.
\end{thm}

\begin{proof}
By the recursive computation of
\begin{equation*}
\int_{2\otimes1}\iota_{2\otimes1}\circ\iota_{1\otimes2}^{-1}(w\otimes\bnull)=\int_{2\otimes1}\iota_{2\otimes1}(\varphi_{k_1,\ldots,k_r}(i,r-i))
\end{equation*}
as above, this is a sum of products of certain MPL2s of main variable $z_2$ and certain MPL1s of $z_1$. Hence \eqref{GHPR''} are relations among multiple polylogarithms of two variables.

In what follows, we prove $\Z$-linear combinations of \eqref{GHPR''}, that is the generalized harmonic product relations for elements of $\Z$-span of $\cMPL$, contain the harmonic product of multiple polylogarithms. Put
\begin{align*}
\cN=\{w\varphi\;|\; &w \text{ is a word of $\xi_1,\xi_{11},\xi_2$ containing at least one $\xi_1$ or $\xi_{11}$,}\\*
&\text{and }\varphi \text{ is an element of $\cMPL$ containing at least one $\xi_{22}$ or $\xi_{12}$}\}.
\end{align*}
For $r>0, s\ge 0$ and $m>0$, applying Proposition \ref{prop:phi_Li_z1} inductively, we obtain
\begin{align}
&\sum_{p=0}^{r-1}\Big(\varphi_{k_1,\ldots,k_{r-p},m,k_{r-p+1},\ldots,k_{r+s}}(r-p,s+p+1) \label{thm_GHPR_HP_proof1}\\*[-3ex]
&\hspace{2cm}+\varphi_{k_1,\ldots,k_{r-p-1},(k_{r-p}+m),k_{r-p+1},\ldots,k_{r+s}}(r-p-1,s+p+1)\Big)\notag\\
&=(\text{terms in $\cN$})\notag\\
&\phantom{=}+\underbrace{\xi_{2}\cdots\xi_{2}}_{m-1\text{ times}}\xi_{22}\varphi_{k_1,\ldots,k_{r+s}}(r,s)-\underbrace{\xi_{2}\cdots\xi_{2}}_{m-1\text{ times}}\xi_{22}\varphi_{k_1,\ldots,k_{r+s}}(0,r+s), \notag
\end{align}
where $(k_1,\ldots,k_0)=\emptyset$. Since each term in $\cN$ has at least one $\xi_1$ or $\xi_{11}$ on the left side of $\xi_2,\xi_{22},\xi_{12}$, we have
\begin{equation*}
\int_{C_{2\otimes1}}(\text{terms in $\cN$})=0.
\end{equation*}
Hence the generalized harmonic product relation for \eqref{thm_GHPR_HP_proof1} is expressed as
\begin{align*}
&\sum_{p=0}^{r-1}\Big(\Li_{k_1,\ldots,k_{r-p},m,k_{r-p+1},\ldots,k_{r+s}}(r-p,s+p+1)\\*[-3ex]
&\hspace{2cm}+\Li_{k_1,\ldots,k_{r-p-1},(k_{r-p}+m),k_{r-p+1},\ldots,k_{r+s}}(r-p-1,s+p+1)\Big)\\
&=\int_{C_{2\otimes1}}\underbrace{\xi_{2}\cdots\xi_{2}}_{m-1\text{ times}}\xi_{22}\varphi_{k_1,\ldots,k_{r+s}}(r,s)-\int_{C_{2\otimes1}}\underbrace{\xi_{2}\cdots\xi_{2}}_{m-1\text{ times}}\xi_{22}\varphi_{k_1,\ldots,k_{r+s}}(0,r+s)\\
&=\int_{C_{2\otimes1}}\underbrace{\xi_{2}\cdots\xi_{2}}_{m-1\text{ times}}\xi_{22}\varphi_{k_1,\ldots,k_{r+s}}(r,s)-\Li_{m,k_1,\ldots,k_{r+s}}(1,r+s;z_2,z_1).
\end{align*}

Therefore the relation
\begin{align}
&\int_{C_{2\otimes1}} \underbrace{\xi_{2}\cdots\xi_{2}}_{m-1\text{ times}}\xi_{22}\Li_{k_1,\ldots,k_{r+s}}(r,s;z_1,z_2) \label{thm_GHPR_HP_proof2}\\
&=\int_{2\otimes1} \iota_{2\otimes1}^{-1}(\underbrace{\xi_{2}\cdots\xi_{2}}_{m-1\text{ times}}\xi_{22}\varphi_{k_1,\ldots,k_{r+s}}(r,s))\notag \\
&= \sum_{p=0}^{r-1} \Big(\Li_{k_1,\ldots,k_{r-p},m,k_{r-p+1},\ldots,k_{r+s}}(r-p,s+p+1;z_1,z_2)\notag \\[-2ex]
&  \hspace{1cm}+\Li_{k_1,\ldots,k_{r-p-1},(k_{r-p}+m),k_{r-p+1},\ldots,k_{r+s}}(r-p-1,s+p+1;z_1,z_2)\Big)\notag \\
& \hspace{0.5cm} +\Li_{m,k_1,\ldots,k_{r+s}}(1,r+s;z_2,z_1) \notag
\end{align}
for $r>0,s \ge 0, m>0$ is the generalized harmonic product relation for some element of $\Z$-span of $\cMPL$.\\

In the last of this proof, we calculate
\begin{equation}\label{thm_GHPR_HP_proof3}
\int_{C_{2\otimes1}}\underbrace{\xi_{2}\cdots\xi_{2}}_{l_1-1\text{ times}}\xi_{22}\cdots\underbrace{\xi_{2}\cdots\xi_{2}}_{l_s-1\text{ times}}\xi_{22}\Li_{k_1,\ldots,k_r}(r,0;z_1,z_2)
\end{equation}
by using \eqref{thm_GHPR_HP_proof2} recursively and show that the result, which is also the generalized harmonic product relations for some element of $\Z$-span of $\cMPL$, is the harmonic product relation \eqref{MPL-HP-RD}. Since
\begin{multline*}
\int_{C_{2\otimes1}}\underbrace{\xi_{2}\cdots\xi_{2}}_{l_1-1\text{ times}}\xi_{22}\cdots\underbrace{\xi_{2}\cdots\xi_{2}}_{l_s-1\text{ times}}\xi_{22}\Li_{k_1,\ldots,k_r}(r,0;z_1,z_2)\\
=\Li_{l_1,\ldots,l_s}(z_2)\Li_{k_1,\ldots,k_r}(z_1)
\end{multline*}
is clear, it suffices to prove that the computation of \eqref{thm_GHPR_HP_proof3} by using \eqref{thm_GHPR_HP_proof2} is equal to the right hand side of \eqref{MPL-HP-RD} up to the generalized harmonic product relation for some element of $\Z$-span of $\cMPL$. We show this by induction on $s$.

For $s=1$, we have
\allowdisplaybreaks
\begin{align*}
&\int_{C_{2\otimes1}}\underbrace{\xi_{2}\cdots\xi_{2}}_{l_1-1\text{ times}}\xi_{22}\Li_{k_1,\ldots,k_r}(r,0;z_1,z_2)\\
&=\sum_{p=0}^{r-1} \Big(\Li_{k_1,\ldots,k_{r-p},l_1,k_{r-p+1},\ldots,k_{r}}(r-p,p+1;z_1,z_2) \\*[-2ex]
&  \hspace{3cm}+\Li_{k_1,\ldots,k_{r-p-1},k_{r-p}+l_1,k_{r-p+1},\ldots,k_{r}}(r-p-1,p+1;z_1,z_2)\Big) \\*
& \hspace{2cm} +\Li_{l_1,k_1,\ldots,k_{r}}(1,r;z_2,z_1)\\
&=\sum_{p=1}^{r} \Li_{k_1,\ldots,k_{p},l_1,k_{p+1},\ldots,k_{r}}(p,r-p+1;z_1,z_2)\\*
&\qquad +\sum_{p=1}^{r}\Li_{k_1,\ldots,k_{p-1},k_p+l_1,k_{p+1},\ldots,k_r}(p-1,r-p+1;z_1,z_2)\\*
&\qquad +\Li_{l_1,k_1,\ldots,k_r}(1,r;z_2,z_1)\\
&=\sum_{p=1}^{r-1} \Li_{k_1,\ldots,k_{p},l_1,k_{p+1},\ldots,k_{r}}(p,r-p+1;z_1,z_2)\\*
&\qquad +\Li_{k_1,\ldots,k_{r},l_1}(r,1;z_1,z_2)\\*
&\qquad +\sum_{p=2}^{r}\Li_{k_1,\ldots,k_{p-1},k_p+l_1,k_{p+1},\ldots,k_r}(p-1,r-p+1;z_1,z_2)\\*
&\qquad +\Li_{k_1+l_1,k_2,\ldots,k_r}(0,r;z_1,z_2)\\*
&\qquad +\Li_{l_1,k_1,\ldots,k_r}(1,r;z_2,z_1).
\end{align*}
This is nothing but the right hand side of the equation \eqref{MPL-HP-RD} for $s=1$ up to 
\begin{equation*}
\Li_{k_1+l_1,k_2,\ldots,k_r}(0,r;z_1,z_2)=\Li_{k_1+l_1,k_2,\ldots,k_r}(0,r;z_2,z_1),
\end{equation*}
which is the generalized harmonic product relation for
\begin{equation*}
\iota_{1\otimes2}^{-1}(\xi_1^{k_1+l_1-1}\xi^{(1)}_{12}\xi_1^{k_2-1}\xi^{(1)}_{12}\cdots\xi_1^{k_r-1}\xi^{(1)}_{12}\otimes\bnull) \in \cMPL.\\
\end{equation*}
\allowdisplaybreaks[0]

For general, we assume that
\begin{align}
\int_{C_{2\otimes1}}&\underbrace{\xi_{2}\cdots\xi_{2}}_{l_2-1\text{ times}}\xi_{22}\cdots\underbrace{\xi_{2}\cdots\xi_{2}}_{l_s-1\text{ times}}\xi_{22}\Li_{k_1,\ldots,k_r}(r,0;z_1,z_2) \label{thm_GHPR_HP_proof4} \\*
=&\sum_{p=1}^{r-1} \Big(\Li_{(k_1,\ldots,k_p,l_2)\cdot((k_{p+1},\ldots,k_r)*(l_3,\ldots,l_s))}(p,\bullet;z_1,z_2) \notag \\*[-3ex]
&\qquad\qquad +\Li_{(k_1,\ldots,k_p,k_{p+1}+l_2)\cdot((k_{p+2},\ldots,k_r)*(l_3,\ldots,l_s))}(p,\bullet;z_1,z_2) \Big) \notag \\*
&\qquad +\Li_{(k_1,\ldots,k_r,l_2,\ldots,l_s)}(r,s-1;z_1,z_2) \notag \\*
&+\sum_{p=2}^{s-1} \Big(\Li_{(l_2,\ldots,l_p,k_1)\cdot((k_2,\ldots,k_r)*(l_{p+1},\ldots,l_s))}(p-1,\bullet;z_2,z_1) \notag \\*[-3ex]
&\qquad\qquad +\Li_{(l_2,\ldots,l_p,k_1+l_{p+1})\cdot((k_2,\ldots,k_r)*(l_{p+2},\ldots,l_s))}(p-1,\bullet;z_2,z_1) \Big) \notag \\*
&\qquad +\Li_{(l_2,\ldots,l_s,k_1,\ldots,k_r)}(s-1,r;z_2,z_1) \notag \\*
&+\Li_{(k_1+l_2)\cdot((k_2,\ldots,k_r)*(l_3,\ldots,l_s))}(0,\bullet;z_2,z_1). \notag
\end{align}
We divide the right hand side of \eqref{thm_GHPR_HP_proof4} and denote it by $f_1(z_1,z_2)+f_2(z_1,z_2)+f_3(z_1,z_2)$,
\begin{align*}
f_1(z_1,z_2)&=\sum_{p=1}^{r-1} \Big(\Li_{(k_1,\ldots,k_p,l_2)\cdot((k_{p+1},\ldots,k_r)*(l_3,\ldots,l_s))}(p,\bullet;z_1,z_2) \notag \\*[-3ex]
&\qquad\qquad +\Li_{(k_1,\ldots,k_p,k_{p+1}+l_2)\cdot((k_{p+2},\ldots,k_r)*(l_3,\ldots,l_s))}(p,\bullet;z_1,z_2) \Big) \notag \\*
&\qquad +\Li_{(k_1,\ldots,k_r,l_2,\ldots,l_s)}(r,s-1;z_1,z_2), \notag \\
f_2(z_1,z_2)&=\sum_{p=2}^{s-1} \Big(\Li_{(l_2,\ldots,l_p,k_1)\cdot((k_2,\ldots,k_r)*(l_{p+1},\ldots,l_s))}(p,\bullet;z_2,z_1) \notag \\*[-3ex]
&\qquad\qquad +\Li_{(l_2,\ldots,l_p,k_1+l_{p+1})\cdot((k_2,\ldots,k_r)*(l_{p+2},\ldots,l_s))}(p,\bullet;z_2,z_1) \Big) \notag \\*
&\qquad +\Li_{(l_2,\ldots,l_s,k_1,\ldots,k_r)}(s-1,r;z_2,z_1), \notag \\
f_3(z_1,z_2)&=\Li_{(k_1+l_2)\cdot((k_2,\ldots,k_r)*(l_3,\ldots,l_s))}(0,\bullet;z_2,z_1).
\end{align*}

Then we have
\allowdisplaybreaks
\begin{align*}
&\int_{C_{2\otimes1}}\underbrace{\xi_{2}\cdots\xi_{2}}_{l_1-1\text{ times}}\xi_{22}\cdots\underbrace{\xi_{2}\cdots\xi_{2}}_{l_s-1\text{ times}}\xi_{22}\Li_{k_1,\ldots,k_r}(r,0;z_1,z_2)\\
&=\int_{C_{2\otimes1}}\underbrace{\xi_{2}\cdots\xi_{2}}_{l_1-1\text{ times}}\xi_{22}\; (\text{RHS of \eqref{thm_GHPR_HP_proof4}})\\
&=\int_{C_{2\otimes1}}\underbrace{\xi_{2}\cdots\xi_{2}}_{l_1-1\text{ times}}\xi_{22}\; f_1(z_1,z_2) + \int_{C_{2\otimes1}}\underbrace{\xi_{2}\cdots\xi_{2}}_{l_1-1\text{ times}}\xi_{22}\; f_2(z_1,z_2)\\
&\hspace{3cm} + \int_{C_{2\otimes1}}\underbrace{\xi_{2}\cdots\xi_{2}}_{l_1-1\text{ times}}\xi_{22}\; f_3(z_1,z_2).\\
\end{align*}
\allowdisplaybreaks[0]

For $f_2(z_1,z_2)$ and $f_3(z_1,z_2)$,
\begin{align*}
\int_{C_{2\otimes1}}&\underbrace{\xi_{2}\cdots\xi_{2}}_{l_1-1\text{ times}}\xi_{22}\; f_2(z_1,z_2)\\
&=\sum_{p=2}^{s-1} \Big(\Li_{(l_1,\ldots,l_p,k_1)\cdot((k_2,\ldots,k_r)*(l_{p+1},\ldots,l_s))}(p,\bullet;z_2,z_1)\\*[-1ex]
&\phantom{+\sum_{p=2}^s\Big(} \quad +\Li_{(l_1,\ldots,l_p,k_1+l_{p+1})\cdot((k_2,\ldots,k_r)*(l_{p+2},\ldots,l_s))}(p,\bullet;z_2,z_1)\Big) \\*
&\phantom{= }+\Li_{l_1,\ldots,l_s,k_1,\ldots,k_r}(s,r;z_2,z_1)
\intertext{and}
\int_{C_{2\otimes1}}&\underbrace{\xi_{2}\cdots\xi_{2}}_{l_1-1\text{ times}}\xi_{22}\; f_3(z_1,z_2)=\Li_{(l_1,(k_1+l_2))\cdot((k_2,\ldots,k_r)*(l_3,\ldots,l_s))}(1,\bullet;z_2,z_1)
\end{align*}
are clear by direct computation. For $f_1(z_1,z_2)$, by using \eqref{thm_GHPR_HP_proof2}, we have
\allowdisplaybreaks
\begin{align*}
&\int_{C_{2\otimes1}}\underbrace{\xi_{2}\cdots\xi_{2}}_{l_1-1\text{ times}}\xi_{22}\; f_1(z_1,z_2)\\
&=\sum_{p=1}^{r-1}\sum_{q=p}^{r-1}\Big(\Li_{(k_1,\ldots,k_p,l_1,k_{p+1},\ldots,k_q,l_2)\cdot((k_{q+1},\ldots,k_r)*(l_3,\ldots,l_s))}(p,\bullet;z_1,z_2)\\*[-2ex]
&\phantom{=+\sum_{p=1}^{r-1}} +\Li_{(k_1,\ldots,k_p,l_1,k_{p+1},\ldots,k_q,k_{q+1}+l_2)\cdot((k_{q+2},\ldots,k_r)*(l_3,\ldots,l_s))}(p,\bullet;z_1,z_2)\Big)\\*[-2ex]
&\phantom{= }+\sum_{p=1}^{r} \Li_{k_1,\ldots,k_p,l_1,k_{p+1},\ldots,k_r,l_2,\ldots,l_s}(p,\bullet;z_1,z_2)\\*
&\phantom{= }+\!\sum_{p=0}^{r-2}\sum_{q=p+1}^{r-1}\!\! \Big(\!\Li_{(k_1,\ldots,k_p,k_{p+1}+l_1,k_{p+2},\ldots,k_q,l_2)\cdot((k_{q+1},\ldots,k_r)*(l_3,\ldots,l_s))}(p,\!\bullet;z_1,z_2)\\*[-2ex]
&\phantom{=+\sum_{p=1}^{r-2}}\!\!\!\!\! +\Li_{(k_1,\ldots,k_p,k_{p+1}+l_1,k_{p+2},\ldots,k_q,k_{q+1}+l_2)\cdot((k_{q+2},\ldots,k_r)*(l_3,\ldots,l_s))}(p,\!\bullet;z_1,z_2)\Big)\\*[-2ex]
&\phantom{= }+\sum_{p=0}^{r-1} \Li_{k_1,\ldots,k_p,k_{p+1}+l_1,k_{p+2},\ldots,k_r,l_2,\ldots,l_s}(p,\bullet;z_1,z_2)\\*
&\phantom{= }+\sum_{p=1}^{r-1}\Big(\Li_{(l_1,k_1,\ldots,k_p,l_2)\cdot((k_{p+1},\ldots,k_r)*(l_3,\ldots,l_s))}(1,\bullet;z_2,z_1)\\*[-2ex]
&\phantom{=\sum_{p=1}^{r-1}}+\Li_{(l_1,k_1,\ldots,k_p,l_2+k_{p+1})\cdot((k_{p+2},\ldots,k_r)*(l_3,\ldots,l_s))}(1,\bullet;z_2,z_1)\Big)\\*[-1ex]
&\phantom{= }+\Li_{l_1,k_1,\ldots,k_r,l_2,\ldots,l_s}(1,r+s-1;z_2,z_1)\\
&=\sum_{p=1}^{r-1} \Big(\Li_{(k_1,\ldots,k_p,l_1)\cdot((k_{p+1},\ldots,k_r)*(l_2,\ldots,l_s))}(p,\bullet,z_1,z_2)\\*[-4ex]
&\phantom{+\sum_{p=1}^{r-1} \Big(} \quad +\Li_{(k_1,\ldots,k_p,l_1+k_{p+1})\cdot((k_{p+2},\ldots,k_r)*(l_2,\ldots,l_s))}(p,\bullet,z_1,z_2)\Big) \\*[-1ex]
&\phantom{= }+\Li_{k_1,\ldots,k_r,l_1,\ldots,l_s}(r,s;z_1,z_2)\\*
&\phantom{= }+\Li_{(l_1,k_1)\cdot((k_2,\ldots,k_r)*(l_{2},\ldots,l_s))}(1,\bullet;z_2,z_1)\\*
&\phantom{= }+\Li_{(l_1+k_1)\cdot((k_2,\ldots,k_r)*(l_2,\ldots,l_s))}(0,\bullet;z_1,z_2).
\end{align*}
\allowdisplaybreaks[0]
The last equation is due to calculations by using identities on $\cIDX$
\begin{align*}
(k_1,\ldots,k_r)&*(l_1,\ldots,l_s)\\
&=\sum_{p=0}^{r-1}\Big((k_1,\ldots,k_p,l_1)\cdot((k_{p+1},\ldots,k_r)*(l_2,\ldots,l_s))\\[-3ex]
&\phantom{=\sum_{p=0}^{r-1}\Big(}+(k_1,\ldots,k_p,k_{p+1}+l_1)\cdot((k_{p+2},\ldots,k_r)*(l_2,\ldots,l_s))\Big)\\
&\qquad +(k_1,\ldots,k_r,l_1,\ldots,l_s),
\end{align*}
which is proved by easy induction on $r$.

Thus we have
\begin{equation*}
\int_{C_{2\otimes1}}\underbrace{\xi_{2}\cdots\xi_{2}}_{l_1-1\text{ times}}\xi_{22}\; (\text{RHS of \eqref{thm_GHPR_HP_proof4}})=\text{RHS of \eqref{MPL-HP-RD}}.\end{equation*}

We have proved the theorem.
\end{proof}

This theorem says that the harmonic product of multiple polylogarithms is regarded as a special case of the generalized harmonic product relations. We note that the generalized harmonic product relations contain the harmonic product for multiple polylogarithms as a proper subset. Indeed, the generalized harmonic product relation for $\iota_{1\otimes2}^{-1}(\xi_{12}^{(1)}\xi_{11}\otimes\bnull)$ is
\begin{align*}
L({}^{1}z_2{}^{1}1;z_1)&=\Li_1(0,1;z_2,z_1)\Li_1(z_1)-\Li_1(z_2)\Li_1(z_1)\\
&\qquad \qquad +\Li_{1,1}(1,1;z_2,z_1)+\Li_{2}(0,1;z_2,z_1),
\end{align*}
which is not the harmonic product of multiple polylogarithms. The hierarchy of the generalized harmonic product relations is interpreted as follows:
\begin{align*}
&\quad\framebox{\text{The generalized harmonic product relations for arbitrary $\varphi \in \cB^0$}}\\[1ex]
\supset &\quad\framebox{$\begin{array}{ll}\text{The generalized harmonic product relations for}\\ \text{$\varphi \in \text{$\Z$-span of }\{\iota_{1\otimes2}^{-1}(w\otimes\bnull)\}$, where $w \in S^0(\xi_1,\xi_{11},\xi^{(1)}_{12})$}\end{array}$}\\[1ex]
\supsetneq &\quad\framebox{\text{The generalized harmonic product relations for $\varphi \in \text{$\Z$-span of }\cMPL$}}\\[1ex]
\supset &\quad \framebox{\text{The harmonic product of multiple polylogarithms}}\end{align*}

\section{The fundamental solution of the KZ equation of two variables and the decomposition}

In this section, we construct the fundamental solution normalized at the origin of KZE2 and show the decomposition theorem of the fundamental solution.\\

Let $\varOmega_0, \varOmega'$ be connection forms defined as
\begin{align}
\varOmega_0&=\xi_1 X_1+\xi_2 X_2 \quad &\text{(singular part of $\varOmega$ at $(0,0)$)},\\
\varOmega'&=\xi_{11} X_{11}+\xi_{22} X_{22}+\xi_{12} X_{12} &\text{(regular part of $\varOmega$ at $(0,0)$).}
\end{align}

\begin{prop}\label{prop:IISol_2KZ}
KZE2 \eqref{2KZeq} has the unique solution $\cL(z_1,z_2)$ satisfying the asymptotic condition $\cL(z_1,z_2)=\hcL(z_1,z_2) z_1^{X_1} z_2^{X_2}$, $\hcL(z_1,z_2)$ is holomorphic at $(0,0)$ and $\hcL(0,0)=\bunit$. This solution is referred to as the fundamental solution normalized at $(0,0)$ and expressed as
\begin{align}
\hcL(z_1,z_2)&=\sum_{s=0}^\infty \hcL_s(z_1,z_2), \label{iterated_integral_form} \\
\hcL_s(z_1,z_2)&=\int_{(0,0)}^{(z_1,z_2)} \left(\ad(\varOmega_0)+\mu(\varOmega')\right)^s(\bnull \otimes \bunit), \notag
\end{align}
where $\ad$ $($resp. $\mu$$)$ stands for the adjoint action $($resp. the left multiplication$)$ for the $\cU(\fX)$-component, that is,
\begin{align*}
\ad(\omega\otimes X)(\varphi\otimes F)&=\omega\varphi \otimes [X,F]\\
(\text{resp. } \mu(\omega\otimes X)(\varphi\otimes F)&=\omega\varphi \otimes XF)
\end{align*}
for $\omega \in \{\xi_1,\ldots,\xi_{12}\},\; \varphi \in \cB,\; X \in \{X_1,\ldots,X_{12}\},\; F \in \cU(\fX)$.

Furthermore, for all $s$, the integrand form $\left(\ad(\varOmega_0)+\mu(\varOmega')\right)^s(\bnull \otimes \bunit)$ belongs to $\cB_s^0 \otimes \cU_s(\fX)$.
\end{prop}

\begin{proof}
If the solution exists, the holomorphic part $\ds \hcL(z_1,z_2)$ satisfies a differential equation
\begin{equation*}
d\hcL(z_1,z_2)=[\varOmega_0,\hcL(z_1,z_2)]+\varOmega'\hcL(z_1,z_2).
\end{equation*}
This is equivalent to the recursive equation 
\begin{equation*}
d\hcL_{s+1}(z_1,z_2)=[\varOmega_0,\hcL_s(z_1,z_2)]+\varOmega'\hcL_s(z_1,z_2)
\end{equation*}
for all $s$. Thus if the iterated integral $\ds \int_{(0,0)}^{(z_1,z_2)}\left(\ad(\varOmega_0)+\mu(\varOmega')\right)^s(\bnull \otimes \bunit)$ is well-defined as a many-valued analytic function on $\cM_{0,5}$, the integral
\begin{equation*}
\hcL_{s}(z_1,z_2)=\int_{(0,0)}^{(z_1,z_2)}\left(\ad(\varOmega_0)+\mu(\varOmega')\right)^s(\bnull \otimes \bunit)
\end{equation*}
determines the unique solution satisfying the asymptotic condition as above. So, for proving this proposition, it is sufficient to prove $\left(\ad(\varOmega_0)+\mu(\varOmega')\right)^s(\bnull \otimes \bunit)$ belongs to $\cB_s^0 \otimes \cU_s(\fX)$.\\

We prove it by induction on $s$. We denote by $P_l$ a map defined by
\begin{equation*}
P_l(\sum_I \omega_{i_1}\cdots\omega_{i_s}\otimes X_I)=\sum_I (\sum_{l=1}^{s-1} \omega_{i_1}\otimes\cdots\otimes\omega_{i_l}\wedge\omega_{i_{l+1}}\otimes\omega_{i_s})\otimes X_I
\end{equation*}
and by $\varOmega^{(s)}=\big(\ad(\varOmega_0)+\mu(\varOmega')\big)^s(\bnull\otimes\bunit)$. For $s=0$ and $1$, it is clear that $\varOmega^{(s)}$ belongs to $\cB_s^0\otimes\cU_s(\fX)$. For $s>1$, since
\begin{equation*}
\varOmega^{(s)}=[\varOmega_0[\varOmega_0,\varOmega^{(s-2)}]]+[\varOmega_0,\varOmega'\varOmega^{(s-2)}]+\varOmega'[\varOmega_0,\varOmega^{(s-2)}]+\varOmega'\varOmega'\varOmega^{(s-2)}
\end{equation*}
and
\begin{multline*}
P_1([\varOmega_0[\varOmega_0,\varOmega^{(s-2)}]])=
P_1([\varOmega_0,\varOmega'\varOmega^{(s-2)}]+\varOmega'[\varOmega_0,\varOmega^{(s-2)}])=P_1(\varOmega'\varOmega'\varOmega^{(s-2)})\\
=0
\end{multline*}
by direct computation, we have $P_1(\varOmega^{(s)})=0$. This implies that $\varOmega^{(s)}$ is an element of $\cB_s^0\otimes\cU_s(\fX)$ by virtue of the hypothesis of the induction.
\end{proof}

\vspace{1\baselineskip}

We restrict KZE2 on the contours $C_{1\otimes2}^{(1)},\; C_{1\otimes2}^{(2)},\; C_{2\otimes1}^{(2)}$ and $C_{2\otimes1}^{(1)}$ which appeared in Section \ref{sec:reduced_bar}. On these contours, KZE2 becomes the following four (generalized) KZ equations of one variable
\begin{alignat}{2}
d_{z_1}G(z_1,z_2)&=\varOmega_{1\otimes2}^{(1)} G(z_1,z_2), \quad &\varOmega_{1\otimes2}^{(1)}&=\xi_1 X_1+\xi_{11} X_{11}+\xi_{12}^{(1)} X_{12}, \label{decomposition_equation_12_1}\\
d_{z_2}G(z_2)&=\varOmega_{1\otimes2}^{(2)} G(z_2), \quad &\varOmega_{1\otimes2}^{(2)}&=\xi_2 X_2+\xi_{22} X_{22}, \label{decomposition_equation_12_2}\\
d_{z_2}G(z_1,z_2)&=\varOmega_{2\otimes1}^{(2)}G(z_1,z_2), \quad &\varOmega_{2\otimes1}^{(2)}&=\xi_2 X_2+\xi_{22} X_{22}+\xi_{12}^{(2)} X_{12}, \label{decomposition_equation_21_2}\\
d_{z_1}G(z_1)&=\varOmega_{2\otimes1}^{(1)} G(z_1), \quad &\varOmega_{2\otimes1}^{(1)}&=\xi_1 X_1+\xi_{11} X_{11}, \label{decomposition_equation_21_1}
\end{alignat}
where $d_{z_1}$ (resp. $d_{z_2}$) stands for the exterior derivation by $z_1$ (resp. $z_2$). The fundamental solution normalized at the origin of each equation is given by
\begin{align*}
\cL_{1\otimes2}^{(1)}&=\hcL_{1\otimes2}^{(1)}\;z_1^{X_1},\\
\cL_{1\otimes2}^{(2)}&=\hcL_{1\otimes2}^{(2)}\;z_2^{X_2},\\
\cL_{2\otimes1}^{(2)}&=\hcL_{2\otimes1}^{(2)}\;z_2^{X_2},\\
\cL_{2\otimes1}^{(1)}&=\hcL_{2\otimes1}^{(1)}\;z_1^{X_1}
\end{align*}
respectively, where each $\hcL_{i\otimes j}^{(k)}$ is a holomorphic function at $z_{k}=0$ and satisfies $\hcL_{i\otimes j}^{(k)}\Big|_{z_{k}=0}=\bunit$.

\begin{prop}[The decomposition theorem of the fundamental solution]\label{prop:decomposition}
The fundamental solution $\cL(z_1,z_2)$ normalized at the origin of KZE2 \eqref{2KZeq} is decomposed to the product of fundamental solutions of one variable equations as
\begin{equation}
\cL(z_1,z_2) =\cL_{1 \otimes 2}^{(1)}\cL_{1 \otimes 2}^{(2)} = \cL_{2 \otimes 1}^{(2)}\cL_{2 \otimes 1}^{(1)}. \label{decomposition}
\end{equation}
\end{prop}

\begin{proof}
We show the decomposition $\cL(z_1,z_2)=\cL_{1 \otimes 2}^{(1)}\cL_{1 \otimes 2}^{(2)}$. Since $\hcL(0,z_2)z_2^{X_2}$ is the fundamental solution normalized at $z_2=0$ of the differential equation \eqref{decomposition_equation_12_2}
\begin{equation*}
d_{z_2}G=(\xi_2X_2+\xi_{22}X_{22})G,
\end{equation*}
we have $\hcL(0,z_2)z_2^{X_2}=\cL_{1 \otimes 2}^{(2)}(z_2)$. Put

\begin{equation*}
G(z_1,z_2)=\cL(z_1,z_2)\left(\cL_{1 \otimes 2}^{(2)}(z_2)\right)^{-1}.
\end{equation*}
Then $G(z_1,z_2)$ satisfies the equation \eqref{decomposition_equation_12_1}
\begin{equation*}
d_{z_1}G(z_1,z_2)=\left(\xi_1 X_1+\xi_{11} X_{11}+\xi_{12}^{(1)} X_{12}\right)G(z_1,z_2).
\end{equation*}
Furthermore, from $[X_1,X_2]=[X_1,X_{22}]=0$, we have
\begin{align*}
G(z_1,z_2)&=\hG(z_1,z_2)z_1^{X_1},\\
\hG(z_1,z_2)&=\hcL(z_1,z_2)\left(\hcL_{1 \otimes 2}^{(2)}(z_2)\right)^{-1}
\end{align*}
and
\begin{equation*}
\hG(0,z_2)=\hcL(0,z_2)\left(\hcL_{1 \otimes 2}^{(2)}(z_2)\right)^{-1}=\cL_{1 \otimes 2}^{(2)}(z_2)\left(\hcL_{1 \otimes 2}^{(2)}(z_2)\right)^{-1}=\bunit.
\end{equation*}
This asymptotic condition says that $G(z_1,z_2)$ is the fundamental solution normalized at $z_1=0$ of \eqref{decomposition_equation_12_1}. Therefore $G(z_1,z_2)=\cL_{1\otimes2}^{(1)}$ holds. Thus we obtain the decomposition $\cL(z_1,z_2)=\cL_{1 \otimes 2}^{(1)}\cL_{1 \otimes 2}^{(2)}$.
\end{proof}

From $[X_1,X_2]=[X_1,X_{22}]=0$, we have the following corollary immediately.
\begin{cor}
\begin{equation}
\hcL(z_1,z_2) =\hcL_{1 \otimes 2}^{(1)}\hcL_{1 \otimes 2}^{(2)} = \hcL_{2 \otimes 1}^{(2)}\hcL_{2 \otimes 1}^{(1)}. \label{decomposition_holo}
\end{equation}
\end{cor}

\section{The iterated integral expression of the fundamental solution}

In this section, we try to express the fundamental solution \eqref{iterated_integral_form} as iterated integrals along the integral contours $C_{1\otimes2}$ and $C_{2\otimes1}$, and discuss the relationship between the decomposition of the fundamental solution and the generalized harmonic product relations.\\

Let $\alpha: \cU(\fX) \to \End(\cU(\fX))$ be an algebra homomorphism defined as
\begin{equation}
\alpha: (X_1,X_{11},X_2,X_{22},X_{12}) \mapsto (\ad(X_1),\mu(X_{11}),\ad(X_2),\mu(X_{22}),\mu(X_{12}))
\end{equation}
and $\alpha(\bunit)=\id_{\cU}$. We note that $\alpha$ is well-defined as a map on $\cU(\fX)$.

We also define the duality maps
\begin{align*}
\theta_{i \otimes j}^{(i)}:&\quad \cU(\C\{X_{i},X_{ii},X_{12}\}) \to S(\xi_{i},\xi_{ii},\xi_{12}^{(i)}),\\
\theta_{i \otimes j}^{(j)}:&\quad \cU(\C\{X_{j},X_{jj}\}) \to S(\xi_{j},\xi_{jj})
\end{align*}
as
\begin{align}
\theta_{i \otimes j}^{(i)}(X_i)&=\xi_{i},\; &\theta_{i \otimes j}^{(i)}(X_{ii})&=\xi_{ii},\; &\theta_{i \otimes j}^{(i)}(X_{12})&=\xi_{12}^{(i)},\\
\theta_{i \otimes j}^{(j)}(X_j)&=\xi_{j},\; &\theta_{i \otimes j}^{(j)}(X_{jj})&=\xi_{jj}. \notag
\end{align}

For $i_1,\ldots,i_k \in \{1,11,2,22,12\}$, we denote by $\cW^0(X_{i_1},\ldots,X_{i_k})$ the set of all words of letters $X_{i_1},\ldots,X_{i_k}$ ended with other than $X_1$ and $X_2$, and by $|W|$ the length of the word $W$. Since $[X_1,X_2]=[X_1,X_{22}]=0$, we remark that
\begin{equation*}
\alpha(W_1W_2)(\bunit)=\alpha(W_1)(\bunit)\alpha(W_2)(\bunit)
\end{equation*}
holds for $W_1 \in \cW^0(X_1,X_{11},X_{12})$ and $W_2 \in \cW^0(X_2,X_{22})$.\\

Now we calculate the fundamental solution $\cL(z_1,z_2)=\hcL(z_1,z_2)z_1^{X_1}z_2^{X_2}$ normalized at the origin by integrating along the contour $C_{1\otimes2}$ and $C_{2\otimes1}$.

\begin{prop}\label{prop:IISol_2KZ_iota}
We have
\begin{multline}\label{IISol_2KZ_iota12}
(\iota_{1\otimes2}\otimes \id_{\cU})((\ad(\Omega_0)+\mu(\Omega'))^s(\bnull\otimes\bunit))\\
=\sum_{\substack{W_1\in \cW^0(X_1,X_{11},X_{12})\\W_2 \in \cW^0(X_2,X_{22})\\|W_1|+|W_2|=s}} \!\!\! \Big(\theta_{1\otimes 2}^{(1)}(W_1)\otimes\theta_{1\otimes 2}^{(2)}(W_2)\Big) \otimes \alpha(W_1W_2)(\bunit)
\end{multline}
and
\begin{multline}\label{IISol_2KZ_iota21}
(\iota_{2\otimes1}\otimes \id_{\cU})((\ad(\Omega_0)+\mu(\Omega'))^s(\bnull\otimes\bunit))\\
=\sum_{\substack{W_2\in \cW^0(X_2,X_{22},X_{12})\\W_1 \in \cW^0(X_1,X_{11})\\|W_2|+|W_1|=s}} \!\!\! \Big(\theta_{2\otimes 1}^{(2)}(W_2)\otimes\theta_{2\otimes 1}^{(1)}(W_1)\Big) \otimes \alpha(W_2W_1)(\bunit).
\end{multline}
\end{prop}

\begin{proof}
We prove the first equation by induction on $s$. For $s=1$, we have
\begin{align*}
(\iota_{1\otimes2}\otimes \id_{\cU})&((\ad(\Omega_0)+\mu(\Omega'))(\bnull\otimes\bunit))\\
&=(\zeta_{11}\otimes\bnull)\otimes X_{11}+(\zeta^{(1)}_{12}\otimes\bnull)\otimes X_{12}+(\bnull\otimes\zeta_{22})\otimes X_{22}\\
&=(\theta_{1\otimes2}^{(1)}(X_{11})\otimes\theta_{1\otimes2}^{(2)}(\bunit))\otimes X_{11}+(\theta_{1\otimes2}^{(1)}(X_{12})\otimes\theta_{1\otimes2}^{(2)}(\bunit))\otimes X_{12}\\
&\phantom{=}\quad +(\theta_{1\otimes2}^{(1)}(\bunit)\otimes\theta_{1\otimes2}^{(2)}(X_{22}))\otimes X_{22}.
\end{align*}

We assume that the equation holds for $s$. By the hypothesis of the induction, we can express $(\ad(\Omega_0)+\mu(\Omega'))^s(\bnull\otimes\bunit)$ as
\begin{equation*}
(\ad(\Omega_0)+\mu(\Omega'))^s(\bnull\otimes\bunit)=\sum_{\substack{W_1,W_2\\|W_1|+|W_2|=s}}\varphi_{W_1,W_2}\otimes\alpha(W_1W_2)(\bunit),
\end{equation*}
where $\varphi_{W_1,W_2}=\iota_{1\otimes2}^{-1}(\theta_{1\otimes 2}^{(1)}(W_1)\otimes\theta_{1\otimes 2}^{(2)}(W_2)) \in \cB^0$. Thus we obtain
\begin{align*}
&(\iota_{1\otimes2}\otimes \id_{\cU})(\ad(\Omega_0)+\mu(\Omega'))^{s+1}(\bnull\otimes\bunit)\\
&=\sum_{\substack{W_1,W_2\\|W_1|+|W_2|=s}}\sum_{i\in\{1,11,2,22,12\}}\iota_{1\otimes2}(\xi_i\varphi_{W_1,W_2})\otimes\alpha(X_iW_1W_2)(\bunit).
\end{align*}
By definition, we have $\iota_{1\otimes2}(\xi_i\varphi_{W_1,W_2})=\theta_{1\otimes2}^{(1)}(X_iW_1)\otimes\theta_{1\otimes2}^{(2)}(W_2)$ for $i=1,11,12$ and $\alpha(X_1\bunit W_2)(\bunit)=0$. On the other hand, if $W_1\neq \bunit$, each term of $\varphi_{W_1,W_2}$ has at least one $\xi_{11}$ or $\xi_{12}$. Thus
\begin{equation*}
\iota_{1\otimes2}(\xi_2\varphi_{W_1,W_2})=\iota_{1\otimes2}(\xi_{22}\varphi_{W_1,W_2})=0
\end{equation*}
for $W_1 \neq \bunit$ holds. Therefore we obtain
\begin{align*}
&(\iota_{1\otimes2}\otimes \id_{\cU})(\ad(\Omega_0)+\mu(\Omega'))^{s+1}(\bnull\otimes\bunit)\\
&=\sum_{\substack{W_1\neq \bunit,W_2\\|W_1|+|W_2|=s}}\big(\theta_{1\otimes2}^{(1)}(X_1W_1)\otimes\theta_{1\otimes2}^{(2)}(W_2)\big)\otimes\alpha(X_1W_1W_2)(\bunit)\\
&\hphantom{=}+\sum_{\substack{W_1,W_2\\|W_1|+|W_2|=s}}\Big(\big(\theta_{1\otimes2}^{(1)}(X_{11}W_1)\otimes\theta_{1\otimes2}^{(2)}(W_2)\big)\otimes\alpha(X_{11}W_1W_2)(\bunit)\\[-4ex]
&\hphantom{=\sum_{\substack{W_1,W_2\\|W_1|+|W_2|=s}}\Big(}\qquad+\big(\theta_{1\otimes2}^{(1)}(X_{12}W_1)\otimes\theta_{1\otimes2}^{(2)}(W_2)\big)\otimes\alpha(X_{12}W_1W_2)(\bunit)\Big)\\
&\phantom{=}+\sum_{\substack{W_2\\|W_2|=s}}\Big(\big(\theta_{1\otimes2}^{(1)}(\bunit)\otimes\theta_{1\otimes2}^{(2)}(X_2W_2)\big)\otimes\alpha(X_2W_2)(\bunit)\\[-4ex]
&\hphantom{=+\sum_{\substack{W_2\\|W_2|=s}}\Big(}\qquad+\big(\theta_{1\otimes2}^{(1)}(\bunit)\otimes\theta_{1\otimes2}^{(2)}(X_{22}W_2)\big)\otimes\alpha(X_{22}W_2)(\bunit)\Big)\\
&=\sum_{\substack{W_1,W_2\\|W_1|+|W_2|=s+1}} \!\!\! \big(\theta_{1\otimes 2}^{(1)}(W_1)\otimes\theta_{1\otimes 2}^{(2)}(W_2)\big) \otimes \alpha(W_1W_2)(\bunit).
\end{align*}

\end{proof}

\begin{cor}\label{cor:decomposition_contour}
The decomposition \eqref{decomposition_holo} of the fundamental solution normalized at the origin
\begin{align*}
\hcL(z_1,z_2)&=\hcL_{1\otimes2}^{(1)}(z_1,z_2)\hcL_{1\otimes2}^{(2)}(z_2)\\
\Big(\text{resp. }\hcL(z_1,z_2)&=\hcL_{2\otimes1}^{(2)}(z_1,z_2)\hcL_{2\otimes1}^{(1)}(z_1)\Big)
\end{align*}
is equal to the iterated integration
\begin{align*}
\hcL(z_1,z_2)&=\sum_{s=0}^\infty \int_{C_{1\otimes2}}\left(\ad(\varOmega_0)+\mu(\varOmega')\right)^s(\bnull \otimes \bunit)\\
\Big(\text{resp. }\hcL(z_1,z_2)&=\sum_{s=0}^\infty \int_{C_{2\otimes1}}\left(\ad(\varOmega_0)+\mu(\varOmega')\right)^s(\bnull \otimes \bunit)\Big)
\end{align*}
along the contour $C_{1\otimes2}$ $($resp. $C_{2\otimes1}$$)$.
\end{cor}

\begin{proof}
From Proposition \ref{prop:IISol_2KZ_iota} and $\alpha(W_1W_2)(\bunit)=\alpha(W_1)(\bunit)\alpha(W_2)(\bunit)$, we have
\allowdisplaybreaks
\begin{align*}
&\int_{C_{1\otimes2}}\left(\ad(\varOmega_0)+\mu(\varOmega')\right)^s(\bnull \otimes \bunit)\\
&=\int_{1\otimes2}(\iota_{1\otimes2}\otimes\id_{\cU})\left(\ad(\varOmega_0)+\mu(\varOmega')\right)^s(\bnull \otimes \bunit)\\
&=\sum_{\substack{W_1\in \cW^0(X_1,X_{11},X_{12})\\W_2 \in \cW^0(X_2,X_{22})\\|W_1|+|W_2|=s}} \int_{1\otimes2}\Big(\theta_{1\otimes 2}^{(1)}(W_1)\otimes\theta_{1\otimes 2}^{(2)}(W_2)\Big) \otimes \alpha(W_1W_2)(\bunit)\\
&=\sum_{\substack{W_1\in \cW^0(X_1,X_{11},X_{12})\\W_2 \in \cW^0(X_2,X_{22})\\|W_1|+|W_2|=s}} \int_{z_1=0}^{z_1}\theta_{1\otimes 2}^{(1)}(W_1)\int_{z_2=0}^{z_2}\theta_{1\otimes 2}^{(2)}(W_2) \alpha(W_1)(\bunit)\alpha(W_2)(\bunit)\\
&=\sum_{s_1+s_2=s}\left(\sum_{\substack{W_1\in \cW^0(X_1,X_{11},X_{12})\\|W_1|=s_1}} \int_{z_1=0}^{z_1}\theta_{1\otimes 2}^{(1)}(W_1)\alpha(W_1)(\bunit)\right)\\*
&\phantom{\sum_{s_1+s_2=s}}\qquad\times  \left(\sum_{\substack{W_2\in \cW^0(X_2,X_{22})\\|W_2|=s_2}} \int_{z_2=0}^{z_2}\theta_{1\otimes 2}^{(2)}(W_2) \alpha(W_2)(\bunit)\right).
\end{align*}
\allowdisplaybreaks[0]
This is nothing but the degree $s$ part of $\hcL_{1\otimes2}^{(1)}(z_1,z_2)\hcL_{1\otimes2}^{(2)}(z_2)$.
\end{proof}

This corollary says that two decompositions in Proposition \ref{prop:decomposition}
\begin{align*}
\hcL(z_1,z_2)=\hcL_{1\otimes2}^{(1)}(z_1,z_2)\hcL_{1\otimes2}^{(2)}(z_2)
\intertext{and}
\hcL(z_1,z_2)=\hcL_{2\otimes1}^{(2)}(z_1,z_2)\hcL_{2\otimes1}^{(1)}(z_1)
\end{align*}
correspond to the choice of integral contours $C_{1\otimes2}$ or $C_{2\otimes1}$ respectively.\\

Finally, we discuss the relationship between the generalized harmonic product relations and the decomposition theorem. For this purpose, we consider the subspace of $\cU(\fX)$ spanned by elements $\alpha(W_1W_2)(\bunit)$ for $W_1 \in \cW^0(X_1,X_{11},X_{12})$, $W_2 \in \cW^0(X_2,X_{22})$.

\begin{prop}\label{prop:cU0}
\begin{enumerate}
\item Two sets
\begin{align*}
&\{\alpha(W_1W_2)(\bunit)\;|\; W_1 \in \cW^0(X_1,X_{11},X_{12}), W_2 \in \cW^0(X_2,X_{22})\},\\
&\{\alpha(W_2W_1)(\bunit)\;|\; W_2 \in \cW^0(X_2,X_{22},X_{12}), W_1 \in \cW^0(X_1,X_{11})\}
\end{align*}
are both linearly independent sets of $\cU(\fX)$.
\item We have
\begin{align*}
&\text{$\C$-span of }\{\alpha(W_1W_2)(\bunit)\;|\; W_1 \in \cW^0(X_1,X_{11},X_{12}), W_2 \in \cW^0(X_2,X_{22})\}\\
&=\text{$\C$-span of }\{\alpha(W_2W_1)(\bunit)\;|\; W_2 \in \cW^0(X_2,X_{22},X_{12}), W_1 \in \cW^0(X_1,X_{11})\}
\end{align*}
as a subspace of $\cU(\fX)$.
\end{enumerate}
\end{prop}

To prove the proposition, we prepare the following lemma. 

\begin{lem}\label{lem:cU0_prepare}
Let $W$ be a word of $X_2,X_{22},X_{12}$. There exists an element $W'$ of $\cU(\C\{X_2,X_{22},X_{12}\})$ such that
\begin{align*}
&X_1\alpha(W)(\bunit)=\alpha(W)(\bunit)X_1+\alpha(W')(\bunit)\\
(\text{resp. } &X_{11}\alpha(W)(\bunit)=\alpha(W)(\bunit)X_{11}+\alpha(W')(\bunit)).
\end{align*}
\end{lem}

\begin{proof}
This can be proved by easy induction on the length of the word $W$. We show the case for $X_1\alpha(W)(\bunit)=\alpha(W)(\bunit)X_1+\alpha(W')(\bunit)$. For $W=\bunit$, we have
\begin{equation*}
X_1\alpha(\bunit)(\bunit)=X_1=\alpha(\bunit)(\bunit)X_1+\alpha(0)(\bunit).
\end{equation*}
We assume that $W=X_2\widetilde{W}$ and the claim holds for $\widetilde{W}$. Since $[X_1,X_2]=0$, we have
\begin{align*}
X_1\alpha(X_2\widetilde{W})(\bunit)&=X_1(X_2\alpha(\widetilde{W})(\bunit)-\alpha(\widetilde{W})(\bunit)X_2)\\
&=X_2X_1\alpha(\widetilde{W})(\bunit)-X_1\alpha(\widetilde{W})(\bunit)X_2\\
&=X_2(\alpha(\widetilde{W})(\bunit)X_1+\alpha(\widetilde{W}')(\bunit))-(\alpha(\widetilde{W})(\bunit)X_1+\alpha(\widetilde{W}')(\bunit))X_2\\
&=X_2\alpha(\widetilde{W})(\bunit)X_1+X_2\alpha(\widetilde{W}')(\bunit)-\alpha(\widetilde{W})(\bunit)X_2X_1-\alpha(\widetilde{W}')(\bunit)X_2\\
&=(X_2\alpha(\widetilde{W})(\bunit)-\alpha(\widetilde{W})(\bunit)X_2)X_1+X_2\alpha(\widetilde{W}')(\bunit)-\alpha(\widetilde{W}')(\bunit)X_2\\
&=\alpha(X_2\widetilde{W})(\bunit)X_1+\alpha(X_2\widetilde{W}')(\bunit)
\end{align*}
where $\widetilde{W}' \in \cU(\C\{X_2,X_{22},X_{12}\})$. We can prove the case for $W=X_{22}\widetilde{W}$ and $W=X_{12}\widetilde{W}$ in a similar way.

\end{proof}

\begin{proof}[Proof of Proposition \ref{prop:cU0}]
\begin{enumerate}
\item We consider the first set.

Let $W_1 \in \cW^0(X_1,X_{11},X_{12})$ and $W_2 \in \cW^0(X_2,X_{22})$. 
From
\begin{equation*}
[X_1,\cW^0(X_2,X_{22})]=0,
\end{equation*}
there exists 
$A_{W_1,W_2}^{(i)} \in \cU(\C\{X_1,X_{11},X_{12}\})$,\; $B_{W_1,W_2}^{(j)} \in \cU(\C\{X_2,X_{22}\})$ such that 
\begin{equation*}
\alpha(W_1W_2)(\bunit)
=W_1W_2+\sum_{\substack{i,j\ge 0\\(i,j)\neq (0,0)}}A_{W_1,W_2}^{(i)} X_1^i B_{W_1,W_2}^{(j)} X_2^j.
\end{equation*}
Since $\cU(\fX)$ is equal to $\cU(\C\{X_1,X_{11},X_{12}\})\otimes \cU(\C\{X_2,X_{22}\})$ as a vector space, the set $\{W_1W_2|W_1 \in \cW^0(X_1,X_{11},X_{12}),\; W_2 \in \cW^0(X_2,X_{22})\}$ is a linearly independent set, so is $\{\alpha(W_1W_2)(\bunit)\}$.
\item For proving the claim, it suffices to prove $\alpha(W_1W_2)(\bunit) \in \text{RHS}$ for $W_1 \in \cW^0(X_1,X_{11},X_{12})$,\; $W_2 \in \cW^0(X_2,X_{22})$. We prove this by induction on $|W_1|$. For $W_1=\bunit$, it is clear. We assume $\alpha(\widetilde{W}_1W_2)(\bunit) \in \text{RHS}$.

If $W_1=X_1\widetilde{W}_1$, we have
\begin{align*}
\alpha(W_1W_2)(\bunit)&=\alpha(X_1\widetilde{W}_1W_2)(\bunit)\\
&=X_1\alpha(\widetilde{W}_1W_2)(\bunit)-\alpha(\widetilde{W}_1W_2)(\bunit)X_1.
\end{align*}
By the hypothesis of the induction, $\alpha(\widetilde{W}_1W_2)(\bunit)$ can be written as
\begin{equation*}
\alpha(\widetilde{W}_1W_2)(\bunit)=\sum_{i}c_i\alpha(W^{(i)}_2W^{(i)}_1)(\bunit)=\sum_{i}c_i\alpha(W^{(i)}_2)(\bunit)\alpha(W^{(i)}_1)(\bunit)
\end{equation*}
where $W^{(i)}_2 \in \cW^0(X_2,X_{22},X_{12})$, $W^{(i)}_1 \in \cW^0(X_1,X_{11})$.

Hence, by virtue of Lemma \ref{lem:cU0_prepare}, we obtain
\begin{align*}
\alpha(W_1W_2)(\bunit)&=\sum_{i}c_iX_1\alpha(W^{(i)}_2)(\bunit)\alpha(W^{(i)}_1)(\bunit)-\sum_{i}c_i\alpha(W^{(i)}_2)(\bunit)\alpha(W^{(i)}_1)(\bunit)X_1\\
&=\sum_{i}c_i(\alpha(W^{(i)}_2)(\bunit)X_1+\alpha(W'^{(i)}_2)(\bunit))\alpha(W^{(i)}_1)(\bunit)\\
&\hspace{1cm}-\sum_{i}c_i\alpha(W^{(i)}_2)(\bunit)\alpha(W^{(i)}_1)(\bunit)X_1\\
&=\sum_{i}c_i\alpha(W^{(i)}_2X_1W^{(i)}_1)(\bunit)+\sum_{i}c_i\alpha(W'^{(i)}_2W^{(i)}_1)(\bunit),
\end{align*}
where $W'^{(i)}_2 \in \cU(\C\{X_2,X_{22},X_{12}\})$. Therefore $\alpha(W_1W_2)(\bunit)$ is an element of \text{RHS}.\\

For $W_1=X_{11}\widetilde{W}_1$, we can prove $\alpha(X_{11}\widetilde{W}_1W_2) \in \text{RHS}$ in the same way. For $W_1=X_{12}\widetilde{W}_1$, the claim is clear. Thus we have proved the proposition.

\end{enumerate}
\end{proof}

We denote by $\cU^0(\fX)$ the subspace of $\cU(\fX)$ appeared in Proposition \ref{prop:cU0} (ii):
\begin{align*}
&\cU^0(\fX)\\
&\;=\text{$\C$-span of }\{\alpha(W_1W_2)(\bunit)\;|\; W_1 \in \cW^0(X_1,X_{11},X_{12}), W_2 \in \cW^0(X_2,X_{22})\}\\
&\;=\text{$\C$-span of }\{\alpha(W_2W_1)(\bunit)\;|\; W_2 \in \cW^0(X_2,X_{22},X_{12}), W_1 \in \cW^0(X_1,X_{11})\}.
\end{align*}
Proposition \ref{prop:cU0} (i) says that the map
\begin{equation*}
\alpha(\bullet)(\bunit):\quad W_1W_2 \mapsto \alpha(W_1W_2)(\bunit)
\end{equation*}
for $W_1 \in \cW^0(X_1,X_{11},X_{12}), W_2 \in \cW^0(X_2,X_{22})$ is a linear isomorphism from 
\begin{equation*}
\text{$\C$-span of }\{W_1W_2\;|\; W_1 \in \cW^0(X_1,X_{11},X_{12}), W_2 \in \cW^0(X_2,X_{22})\}
\end{equation*}
to $\cU^0(\fX)$. Moreover $\cB^0\cong S^0(\xi_1,\xi_{11},\xi^{(1)}_{12})\otimes S^0(\xi_2,\xi_{22})$ holds by virtue of Proposition \ref{prop:decomposition_cB0}. Then we can define an isomorphism of vector spaces
\begin{equation*}
\Theta_{1\otimes2}: \cU^0(\fX)\to \cB^0
\end{equation*}
as
\begin{equation}
\Theta_{1\otimes2}(\alpha(W_1W_2)(\bunit))=\iota_{1\otimes2}^{-1}(\theta_{1\otimes2}^{(1)}(W_1)\otimes\theta_{2\otimes1}^{(2)}(W_2)).
\end{equation}
The isomorphism $\Theta_{2\otimes1}: \cU^0(\fX)\to \cB^0$ is also defined in the same fashion. Under this notation, the equations \eqref{IISol_2KZ_iota12} and \eqref{IISol_2KZ_iota21} can be written as
\begin{align}
(\ad(\Omega_0)+\mu(\Omega'))^s(\bnull\otimes\bunit)&=\hspace{-0.8cm}\sum_{\substack{W_1\in \cW^0(X_1,X_{11},X_{12})\\W_2 \in \cW^0(X_2,X_{22})\\|W_1|+|W_2|=s}} \hspace{-0.8cm} \Theta_{1\otimes2}(\alpha(W_1W_2)(\bunit)) \otimes \alpha(W_1W_2)(\bunit) \label{IISol_2KZ_B_12}\\
&=\hspace{-0.8cm}\sum_{\substack{W_2\in \cW^0(X_2,X_{22},X_{12})\\W_1 \in \cW^0(X_1,X_{11})\\|W_1|+|W_2|=s}} \hspace{-0.8cm} \Theta_{2\otimes1}(\alpha(W_2W_1)(\bunit)) \otimes \alpha(W_2W_1)(\bunit). 
\end{align}
From Proposition \ref{prop:cU0}, one can write uniquely
\begin{equation*}
\alpha(W_1W_2)(\bunit)=\sum_i c_i \alpha(W_2^{(i)}W_1^{(i)})(\bunit)
\end{equation*}
where $c_i \in \C, W_2^{(i)} \in \cW^0(X_2,X_{22},X_{12}), W_1^{(i)} \in \cW^0(X_1,X_{11})$. In this case, we have
\begin{equation}\label{contour_GHPR_prepare}
\Theta_{1\otimes2}(\alpha(W_1W_2)(\bunit))=\sum_i c_i \Theta_{2\otimes1}(\alpha(W_2^{(i)}W_1^{(i)})(\bunit))
\end{equation}
in $\cB^0$. Then we can express the generalized harmonic product relation for the element \eqref{contour_GHPR_prepare} as
\begin{equation}\label{contour_GHPR}
\int_{1\otimes2}\theta_{1\otimes2}^{(1)}(W_1)\otimes\theta_{1\otimes2}^{(2)}(W_2)=\sum_i c_i \int_{2\otimes1}\theta_{2\otimes1}^{(2)}(W_2^{(i)})\otimes\theta_{2\otimes1}^{(1)}(W_1^{(1)}).
\end{equation}
This is nothing but the comparison of the coefficients of $\alpha(W_1W_2)(\bunit)$ on
\begin{equation*}
\int_{C_{1\otimes2}}(\ad(\Omega_0)+\mu(\Omega'))^s(\bnull\otimes\bunit)=\int_{C_{2\otimes1}}(\ad(\Omega_0)+\mu(\Omega'))^s(\bnull\otimes\bunit).
\end{equation*}
On the other hand, since $\Theta_{1\otimes2}$ is an isomorphism, each element of a certain basis of $\cB^0$ is appeared just one time in \eqref{IISol_2KZ_B_12}. From these facts and Corollary \ref{cor:decomposition_contour}, we have the following theorem.

\begin{thm}\label{thm:GHPR_is_decomposition}
The generalized harmonic product relations \eqref{GHPR} are equivalent to the relations comes from the comparison of the coefficients on the decomposition \eqref{decomposition_holo}.
\end{thm}

For example, we calculate the degree two holomorphic part of the fundamental solution $\hcL_2(z_1,z_2)$ along the contour $C_{1\otimes2}$ and $C_{2\otimes1}$. For $C_{1\otimes2}$, we have
\begin{align*}
\hcL_2(z_1,z_2)
&=\Li_{2}(1,0;z_1,z_2) [X_1,X_{11}]
+\Li_{2}(0,1;z_1,z_2) [X_1,X_{12}]\\*
&\quad +\Li_{1,1}(2,0;z_1,z_2) X_{11}^2
+\Li_{1,1}(1,1;z_1,z_2) X_{11} X_{12}\\*
&\quad +L({}^{1}z_2{}^{1}1;z_1) X_{12} X_{11}
+\Li_{1,1}(0,2;z_1,z_2) X_{12}^2\\*
&\quad +\Li_1(1,0;z_1,z_2) \Li_1(z_2) X_{11} X_{22}
+\Li_1(0,1;z_1,z_2) \Li_1(z_2) X_{12} X_{22}\\*
&\quad +\Li_{2}(z_2) [X_2,X_{22}]
+\Li_{1,1}(z_2) X_{22}^2.
\end{align*}
On the other hand, for $C_{1\otimes2}$, we have
\begin{align*}
\hcL_2(z_1,z_2)
&=\Li_{2}(z_1) [X_1,X_{11}]
+\Li_{2}(0,1;z_2,z_1) [X_1,X_{12}]
+\Li_{1,1}(z_1) X_{11}^2\\*
&\quad +\big(-\Li_{2}(0,1;z_2,z_1)-\Li_{1,1}(1,1;z_2,z_1)\\*
&\hspace{4cm}+\Li_1(1,0;z_2,z_1) \Li_1(z_1)\big) X_{11} X_{12}\\*
&\quad +\big(\Li_{2}(0,1;z_2,z_1)+\Li_{1,1}(1,1;z_2,z_1)\\*
&\hspace{1cm}-\Li_1(1,0;z_2,z_1) \Li_1(z_1)+\Li_1(0,1;z_2,z_1) \Li_1(z_1)\big) X_{12} X_{11}\\*
&\quad +\Li_{1,1}(0,2;z_2,z_1) X_{12}^2+\Li_1(1,0;z_2,z_1)\Li_1(z_1) X_{11} X_{22}\\*
&\quad +\big(\Li_{1,1}(1,1;z_2,z_1)+L({}^{1}z_1{}^{1}1;z_2)\big) X_{12} X_{22}\\*
&\quad +\Li_{2}(1,0;z_2,z_1) [X_2,X_{22}]
+\Li_{1,1}(2,0;z_2,z_1) X_{22}^2
\end{align*}
by using the infinitesimal pure braid relation \eqref{IPBR2}.
Thus the coefficient of $X_{11}X_{12}$ of $\hcL(z_1,z_2)$ is
\begin{equation*}
\Li_1(z_2) \Li_1(z_1)=\Li_{1,1}(1,1;z_1,z_2)+\Li_{2}(0,1;z_2,z_1)+\Li_{1,1}(1,1;z_2,z_1)
\end{equation*}
and the coefficient of $X_{12}X_{11}$ is
\begin{multline*}
L({}^{1}z_2{}^{1}1;z_1)\\
=\Li_{2}(0,1;z_2,z_1)+\Li_{1,1}(1,1;z_2,z_1)-\Li_1(z_2)\Li_1(z_1)+\Li_1(0,1;z_2,z_1)\Li_1(z_1).
\end{multline*}
There are nothing but the generalized harmonic product relation for elements $\iota_{1\otimes2}^{-1}(\xi_{11}\xi_{12}^{(1)}\otimes\bnull)$ and $\iota_{1\otimes2}^{-1}(\xi_{12}^{(1)}\xi_{11}\otimes\bnull)$ in $\cB^0$.

\vspace{1cm}

\address

\end{document}